\newif\ifAlesStyle

\AlesStyletrue
\AlesStylefalse

\ifAlesStyle

\documentclass[reqno]{amsart}

\else

\documentclass[reqno,12pt]{amsart}
\usepackage[left=2cm,right=2cm,top=1.5cm,bottom=2cm]{geometry}

\fi

\usepackage{amssymb}

\numberwithin{equation}{section}
\numberwithin{table}{section}
\newtheorem{thm}{Theorem}[section]
\newtheorem{prop}[thm]{Proposition}
\newtheorem{lem}[thm]{Lemma}
\newtheorem{cor}[thm]{Corollary}

\def\id{\operatorname{id}}
\def\aut{\operatorname{Aut}}

\def\chr{\operatorname{char}}

\def\gal{\operatorname{Gal}}

\def\vhi{\varphi}

\def\eps{\varepsilon}

\def\m{^{-1}}

\renewcommand{\ge}{\geqslant}
\renewcommand{\le}{\leqslant}

\newcommand{\STS}{{\rm STS}}
\def\dfrac#1#2{\lower0.15ex\hbox{\large$\frac{#1}{#2}$}}

\newcommand{\mul}[1]{\theta_{#1}}

\newcommand{\cref}[1]{Corollary~$\ref{#1}$}
\newcommand{\lref}[1]{Lemma~$\ref{#1}$}
\newcommand{\pref}[1]{Proposition~$\ref{#1}$}
\newcommand{\tref}[1]{Theorem~$\ref{#1}$}
\newcommand{\cmref}[1]{Claim~$\ref{#1}$}
\newcommand{\eref}[1]{$(\ref{e#1})$}
\newcommand{\secref}[1]{Section~$\ref{#1}$}

\newcommand{\F}{\mathbb F}
\newcommand{\K}{\mathbb K}
\newcommand{\LL}{\mathbb L}
\newcommand{\nsq}{\zeta}
\newcommand{\midl}{\,{\mid}\,}
\newcommand{\agl}{\mathrm A \Gamma \mathrm L}
\newcommand{\agtl}{\mathrm A \Gamma^2 \mathrm L}
\newcommand{\agltw}{\mathrm A \Gamma \mathrm L^{\mathrm{tw}}}

\theoremstyle{remark}
\newtheorem{claim}{Claim}

\begin{document}

\title{Isomorphisms of quadratic quasigroups}
\author{Ale\v s Dr\'apal}
\author{Ian M. Wanless}

\thanks{A.~Dr\'apal supported by INTER-EXCELLENCE project 
LTAUSA19070 M\v SMT Czech Republic}
\address{Department of Mathematics \\ Charles University 
\\ Sokolovsk\'a 83 \\ 186
75 Praha 8, Czech Republic}
\address{School of Mathematics \\ 
Monash University \\
Clayton Vic 3800\\
Australia}

\email{drapal@karlin.mff.cuni.cz}
\email{ian.wanless@monash.edu}

\begin{abstract}
Let $\F$ be a finite field of odd order and
$a,b\in\F\setminus\{0,1\}$ be such that $\chi(a) = \chi(b)$
and $\chi(1-a)=\chi(1-b)$, where $\chi$ is the extended quadratic
character on $\F$.
Let $Q_{a,b}$ be the quasigroup over $\F$ defined by
$(x,y)\mapsto x+a(y-x)$ if $\chi(y-x) \ge 0$, and $(x,y)
\mapsto x+b(y-x)$ if $\chi(y-x) = -1$.
We show that $Q_{a,b} \cong Q_{c,d}$ if and only if $\{a,b\}
= \{\alpha(c),\alpha(d)\}$ for some $\alpha\in \aut(\F)$. We also
characterise $\aut(Q_{a,b})$ and exhibit further properties, including
establishing when $Q_{a,b}$ is a Steiner quasigroup or is commutative,
entropic, left or right distributive, flexible or semisymmetric.
In proving our results we also characterise the minimal subquasigroups
of $Q_{a,b}$.
\end{abstract}



\maketitle

\section{Main results}\label{r}
Throughout, $\F$ will be a finite field of odd order. 
A \emph{quasigroup} $Q$ is a set with a binary operation,
say $\cdot$, such that the equations $x\cdot a = b$ and $a\cdot y=b$
have unique solutions for all $a,b\in Q$.
Let $\chi:\F\to \{-1,0,1\}$ be the extended quadratic character, 
which satisfies $\chi(0)=0$ and sends nonzero squares
and nonsquares to $+1$ and $-1$, respectively. 
For any $a,b \in \F$ there exists an operation $*$ on $\F$
such that 
\begin{equation}\label{er1}
x*y = \begin{cases} x+a(y-x) \ \text{if} \ \chi(y-x) \ge 0, \\
x+b(y-x) \ \,\text{if} \ \chi(y-x) = -1.  \end{cases}
\end{equation}
This operation yields a quasigroup (see e.g.~\cite{Eva18}) if and only if 
\begin{equation}\label{er2}
\chi(a)=\chi(b) \ne 0 \text{\, and \,} 
\chi(1-a) = \chi(1-b) \ne 0.
\end{equation}
If \eref{r2} holds, then the quasigroup given by \eref{r1}
will be denoted by $Q_{a,b}$. A finite quasigroup isomorphic
to a quasigroup $Q_{a,b}$ is said to be \emph{quadratic}.
Note that quadratic quasigroups are \emph{idempotent}, i.e.,~they
satisfy the law $xx = x$.

Quadratic quasigroups have many applications, including the construction
of mutually orthogonal Latin squares \cite{Eva92,Eva18},
atomic Latin squares \cite{cyclatom}, Falconer varieties \cite{AW22},
perfect $1$-factorisations of graphs \cite{AW22,GW20,cyclatom}
and maximally non-associative quasigroups \cite{DW21}. A question
raised by \cite{DW21} was to understand when two quadratic quasigroups
are isomorphic. Our first main result answers this question:

\begin{thm}\label{r1}
Let $Q_{a,b}$ and $Q_{c,d}$ be quadratic quasigroups over
$\F$. Then $Q_{a,b}\cong Q_{c,d}$ if and only if 
there exists $\alpha \in \aut(\F)$ such that
$\{a,b\} = \{\alpha(c),\alpha(d)\}$.
\end{thm}

In \cite{cyclatom} it was noted that quadratic quasigroups have rich
automorphism groups. Our second major goal is to fully understand
these groups. We start by defining the following groups:

\begin{itemize}
\item
$\agl_1(\F)$ is the group of all \emph{affine semilinear
mappings} $x \mapsto \lambda \alpha(x) + \mu$, where
$\lambda\in\F^*$, $\mu \in \F$ and $\alpha \in \aut(\F)$.

\item
$\agl_1(\F\midl \K)$ is the subgroup of $\agl_1(\F)$ in which the 
automorphism $\alpha\in \aut(\F)$ fixes every element of a subfield $\K$ of $\F$ (in other words, $\alpha \in \gal(\F\midl \K)$).

\item $\agtl_1(\F)$ is the subgroup of $\agl_1(\F)$ 
consisting of all maps $x \mapsto \lambda \alpha(x) + \mu$ such that
$\chi(\lambda) = 1$. 

\item $\agtl_1(\F\midl \K) = \agtl_1(\F)\cap \agl_1(\F\midl \K)$.

\end{itemize}

The index of $\agtl_1(\F\midl \K)$ in $\agl_1(\F\midl \K)$
is equal to two. If there exists  a subfield $\LL$ such that $[\K:\LL]=2$,
then it is possible to construct another group of affine
semilinear mappings in which
$\agtl_1(\F\midl \K)$ forms a subgroup of index two.
This group is said to be a \emph{twist} of $\agl_1(\F\midl \K)$.
It is denoted by $\agltw_1(\F\midl \K)$ and consists of
$\agtl_1(\F\midl \K)$ and all mappings
\[ x \mapsto \lambda\alpha(x^\gamma)+\mu, \text{ where }
\chi(\lambda) = - 1,\ \alpha \in \gal(\F\midl \K),\ \mu\in\F 
\text{ and } \gamma = |\LL|.\]

\begin{thm}\label{r2}
Let $Q=Q_{a,b}$ be a quadratic quasigroup over $\F$.
Denote by $\K$ the least subfield of $\F$ that contains
$\{a,b\}$. The automorphism group of $Q$ is equal
to $\agtl_1(\F\midl \K)$ up to these exceptions:
\begin{enumerate}
\item[(i)] If $a=b$, then $\aut(Q) \cong \operatorname{AGL}_k(\K)$, where
$k=[\F:\K]$. The automorphisms of $Q$ are all mappings
  $x\mapsto \sigma(x)+\mu$, where $\mu \in \F$ and $\sigma\colon\F \to \F$
  is a $\K$-linear bijection.
\item[(ii)] If there is an integer $\gamma$ such that
$b=a^\gamma$ and $\gamma^2 = |\K|$,
then $\aut(Q) = \agltw_1(\F\midl \K)$.
\item [(iii)] If $|\F|=7$ and $\{a,b\} = \{3,5\}$, then
$\aut(Q) \cong \operatorname{PSL}_2(7)$.
\end{enumerate}
\end{thm}

The proof of \tref{r2} leads us to examine several varieties of
quasigroup, and it becomes important to understand which quadratic
quasigroups those varieties contain. This leads to our third main
result:

\begin{thm}\label{f10}
Let $Q=Q_{a,b}$ be a quadratic quasigroup over $\F$.
Then
\begin{itemize}
\item[(i)] $Q$ is entropic (i.e.~fulfils the law
$xy\cdot uv = xu\cdot yv$) if and only if $a=b$;
\item[(ii)] $Q$ is left distributive (i.e.~fulfils the
law $x\cdot yz = xy \cdot xz$) if and only if $a=b$;
\item[(iii)] $Q$ is right distributive (i.e.~fulfils the
law $xy \cdot z = xz\cdot yz$) if and only if $a=b$;
\item[(iv)] $Q$ is commutative if and only if $a+b = 1$
 and either $|\F|\equiv3\bmod4$ or $a=b$.
\item[(v)] $Q$ is flexible (i.e.~fulfils the law
$x\cdot yx = xy \cdot x$) if and only if 
$a=b$ or $\chi(a) = \chi(1-a) = 1$ or both $a+b =1$ and $|\F|\equiv3\bmod4$;
\item[(vi)] $Q$ is semisymmetric (i.e.~fulfils the law
$xy \cdot x=y$) if and only if $a^2-a+1=0$ and either $a=b$ or $a+b=1$. 
\item[(vii)] $Q$ is a Steiner quasigroup (i.e.~idempotent,
commutative and semisymmetric)
if and only if either $\chr(\F)=3$ and $a=b=-1$,
or $\chr(\F)>3$, $a+b=ab=1$, and $\chi(a)=\chi(-1)=-1$.
In the latter case, $a\ne b$.
\item[(viii)] $Q$ is isotopic to a group if and only if $a=b$.
\end{itemize}
\end{thm}

Another outcome from our work is a precise characterisation of all minimal
subquasigroups of quadratic quasigroups $Q_{a,b}$. See 
Theorems~\ref{s4} and~\ref{s5}.

Regarding \tref{f10}(i), we
note that entropic quasigroups are also sometimes called \emph{medial}.

Regarding \tref{f10}(vii), we make the following remarks.  If
$\chr(\F) = 3$, then $Q_{a,b}$ is Steiner if and only if
$a=b=-1$. Steiner quadratic quasigroups in characteristic $3$ thus
coincide with affine STSs.  If $\chr(\F) > 3$, then $Q_{a,b}$ is a
Steiner quasigroup if and only if $a+b=1=ab$ and
$\chi(a)=\chi(b)=\chi(-1) =-1$. An easy number theoretical argument
shows that this happens if and only if $|\F| = p^k$ for a prime
$p\equiv 7 \bmod 12$ and odd $k\ge 1$, and $a$ and $b$ are distinct
primitive sixth roots of unity. Blocks of the STS are the sets
$\{u,v,av+bu\}$ where $\chi(v-u) = 1$.  These STSs are known
as \emph{Netto systems} and we refer to the corresponding quasigroups
as \emph{Netto quasigroups}. Robinson~\cite{rob} proved that their
automorphism group is equal to $\agtl_1(\F)$, with the exception of
order $7$, which yields the Fano plane---and thus also \tref{r2}(iii).

Say that a quadratic quasigroup $Q = Q_{a,b}$ is \emph{twisted}
if $b = a^\gamma$ where $\gamma^2$ 
is the order of the least subfield of $\F$ containing the element $a$. 
The exceptional cases of \tref{r2} may thus be labelled
entropic, twisted and Fano. It is immediately clear that 
these are the only cases in which $\aut(Q)$ is $2$-transitive.

Twisted quadratic quasigroups are closely related to quasigroups
constructed from quadratic nearfields.
The axioms of a (left) \emph{nearfield} $(N,+,\circ,0,1)$ stipulate
that $(N,+,0)$ is an abelian group, $(N\setminus\{0\},\circ,1)$
is a group, $0\circ x = 0 = x\circ 0$ for all $x\in N$,
and $x\circ(y+z) = x\circ y + x \circ z$, for all $x,y,z\in N$.
A \emph{quadratic nearfield} is defined over a field $\F_{q^2}$,
where $q$ is a power of an odd prime, by 
\begin{equation}\label{er3}
x\circ y = \begin{cases} xy& \text{if} \ \chi(x) \ge 0; \\
  xy^q & \text{if} \ \chi(x) = -1.
\end{cases}
\end{equation}
With each element $c\notin \{0,1\}$ of a nearfield $N$, 
there may be associated a quasigroup $(N,*_c)$ for which
\begin{equation}\label{er4}
x*_c y =  x +(y-x)\circ c \text{ whenever } x,y \in N.
\end{equation}
Stein \cite{stein} showed that
each of the mappings $x\mapsto \lambda \circ x + \mu$, where
$\lambda \in N\setminus\{0\}$ and $\mu\in N$, is an automorphism
of $(N,*_c)$.

The notation $(\F_{q^2},*_c)$ will always refer to the quasigroup
built by means of \eref{r4} over the quadratic nearfield that is
defined on $\F_{q^2}$ by \eref{r3}. These quasigroups may also
be obtained by means of \eref{r1} as quadratic quasigroups:

\begin{prop}\label{r3}
Suppose that $|\F| = q^2$ and that $a\in \F\setminus\{0,1\}$.
Then $Q_{a,a^q} = (\F_{q^2},*_a)$ and $Q_{a^q,a} = 
(\F_{q^2},*_{a^q})$. The mapping $x\mapsto x^q$ yields
an isomorphism $(\F_{q^2},*_a)\cong (\F_{q^2},*_{a^q})$. 
\end{prop}

\begin{proof} We use $*$ to denote the operation of $Q_{a,a^q}$.
If $y-x$ is a square, then $x*y = x+(y-x)a = x+ (y-x)\circ a= x*_a y$.
If $y-x$ is a nonsquare, then $x*y = x+(y-x)a^q = x+ (y-x)\circ a = x*_a y$,
for all $x,y \in \F$.
It follows that $Q_{a,a^q} = (\F_{q^2},*_a)$, and then by substituting
$a^q$ for $a$ we find that $Q_{a^q,a} = (\F_{q^2},*_{a^q})$.
The fact that $x\mapsto x^q$ is 
an isomorphism $Q_{a,a^q}\cong Q_{a^q,a}$ was shown in \cite{cyclatom}
(and also follows from \pref{f3}). 
\end{proof}

\bigskip

Let us briefly outline the content of the following sections.
\secref{f} consists of straightforward arguments that establish
\tref{f10}. To avoid repeating the same condition let us 
assume in the rest of this overview that 
$Q = Q_{a,b}$ is a quadratic quasigroup
defined on the field $\F = \F_q$ by means of \eref{r1} such that
$a\ne b$ and such that $Q$ is not a Steiner quasigroup.

The main achievement of \secref{ss} is \pref{f7}
which shows that every subquasigroup of $Q$ containing $0$ is closed 
under the addition of $\F$. \secref{e} starts by investigating
the situation when there exists an additive $\vhi \in \aut(Q)$ such that
$\vhi(1)$ is a nonsquare in $\F$. Several
technical results are needed to obtain \pref{e5} by which
the latter condition implies that $Q$ is twisted. That suffices
to prove \tref{r2} for the case of $Q$ being $2$-generated. That
is done in \tref{e7}. The structure of $\aut(Q)$ is then used
to establish, in \tref{e8}, the validity of \tref{r1} when 
$Q$ is $2$-generated.

Assume now that $Q$ is not $2$-generated. Call a subquasigroup
minimal if it consists of more than one element and has no 
proper subquasigroup with more than one element. \secref{s}
is devoted to the description of minimal subquasigroups
of $Q$ and of $2$-generated subquasigroups of $Q$.
This is achieved in Theorems~\ref{s4} and~\ref{s5}. It turns out
that such subquasigroups may be used to get a structure of 
affine lines belonging to a subfield of $\F$. Since an automorphism
of $Q$ has to respect such a structure, it has to be induced
by semilinear mappings (\pref{s8}). It turns out that such a mapping
has to be linear
in many cases (\pref{a2}), and that allows us, by an application
of a theorem of Carlitz, to confirm the structure of $\aut(Q)$
as described in \tref{r2}. Knowledge of $\aut(Q)$ is then used
to prove \tref{r1} in its general form.

\section{Varieties of quadratic quasigroups}\label{f}

This section is primarily aimed at proving \tref{f10}.  The proof is
split between Lemmas~\ref{l:grpbased}--
\ref{f9} below.  Let $Q = Q_{a,b}$ be a quadratic quasigroup over
$\F$.  We wish to give easily checkable conditions on $a,b$ under
which $Q$ is entropic, left or right distributive, commutative,
flexible, semisymmetric or totally symmetric.  Note that totally
symmetric is a term describing the combination of commutative and
semisymmetric. Steiner quasigroups are precisely those that are
idempotent and totally symmetric.

We start with some basic properties of quadratic quasigroups.
If $(Q,*)$ is a quasigroup, then the \emph{opposite} and the
\emph{translate} of $(Q,*)$ are, respectively, the quasigroups
$(Q,\circ)$ and $(Q,\otimes)$ defined by
$x*y=z\Leftrightarrow y\circ x=z\Leftrightarrow z\otimes x=y$.
The following statement, with the
exception of points (vii) and (viii), is immediate from \cite{cyclatom}.

\begin{prop}\label{f3}
Let $Q_{a,b}$ be a quadratic quasigroup over $\F$.
\begin{itemize}
\item[(i)] $Q_{a,b}$ is idempotent.
\item[(ii)] For any $f\in\F$ the map $x\mapsto x+f$ is an 
automorphism of $Q_{a,b}$. 
\item[(iii)] For any nonzero square $c\in\F$ the map 
$x\mapsto cx$ is an automorphism of $Q_{a,b}$. 
\item[(iv)]  
$Q_{a,b}$ is isomorphic to $Q_{b,a}$ by the map 
$x\mapsto\nsq x$, where $\nsq$ is any nonsquare in $\F$.
\item[(v)] 
The opposite quasigroup of $Q_{a,b}$ is $Q_{1-a,1-b}$ if 
$|\F|\equiv1\bmod4$ and $Q_{1-b,1-a}$ if $|\F|\equiv3\bmod4$.
\item[(vi)]
  The translate of $Q_{a,b}$ is $Q_{(a-1)/a,(b-1)/b}$ if $\chi(a)=\chi(-1)$
  and $Q_{(b-1)/b,(a-1)/a}$ if $\chi(a)\ne\chi(-1)$.
\item[(vii)] If $a\ne b$ and $\nsq$ is a nonsquare in $\F$,
then $\nsq(x*y)\ne (\nsq x)*(\nsq y)$ for all distinct $x,y\in \F$.
\item[(viii)] If $\alpha\in \aut(\F)$, then
$\alpha$ induces an automorphism between $Q_{a,b}$ and $Q_{\alpha(a),\alpha(b)}$.
\end{itemize}
\end{prop}

\begin{proof} To prove (vii) consider $x,y \in \F$.
If $\chi(y-x) = 1$, then $\nsq(x*y) = \nsq x + a\nsq(y-x)$,
while $\nsq x * \nsq y = \nsq x + b\nsq(y-x)$. 
If $\chi(y-x) = -1$, then $\nsq(x*y) = \nsq x + b\nsq(y-x)$,
while $\nsq x * \nsq y = \nsq x + a\nsq(y-x)$.

To prove (viii), note that
$\chi(y-x) = \chi(\alpha(y-x)) =\chi(\alpha(y)-\alpha(x))$
for all $x,y \in Q$. Hence $x*y = x+a(y-x)$ in $Q_{a,b}$ 
if and only if 
$$\alpha(x)*\alpha(y) = \alpha(x)+\alpha(a)(\alpha(y)-\alpha(x))
=\alpha(x*y)$$ in $Q_{\alpha(a),\alpha(b)}$.
The argument remains true if $a$ is replaced by $b$, so
$\alpha(x)*\alpha(y) =\alpha(x*y)$ in all cases.
\end{proof}

\begin{lem}\label{l:runnonres}
  Suppose $\F$ is a finite field of odd order $|\F|>9$. Then there exist 
  $u,v\in\F$ such that
  $\chi(u)=\chi(v)=\chi(u+1)=\chi(v+1)=\chi(u-1)=-1$ and $\chi(v-1)=1$.
\end{lem}

\begin{proof}
The statement is concerned with two special cases of a more general
problem that asks if for $\eps_i\in \{-1,1\}$, $-1\le i \le 1$, there
exists $x \in \F$ such that $\chi(x+i) = \eps_i$. A consequence
of Weil's bound (e.g., as stated in \cite[Theorem 1.6]{dwa}) implies
that such an $x$ exists if 
\[ 2^{-3}q -(3/2)(\sqrt q + 1) + \sqrt q(1-2^{-3}) > 0, \]
where $q= |\F|$. This is true for each prime power $q > 43$. For prime values
$q = 11$, $13$, $17$, $19$, $23$, $29$, $31$, $37$, $41$ and $43$,
put $u = 7$, $6$, $6$, $13$, $20$, $11$, $12$, $14$, $12$
and $19$, respectively. For $q = 25$ set $u = 2\sqrt 2$. In all these
cases set $v = u-1$. For $q=27$ set $u=x$ and $v=2x^2$ in $\F_3[x]/(x^3+2x+1)$.
\end{proof}

\begin{lem}\label{l:grpbased}
  Suppose that $Q=Q_{a,b}$ is a quadratic quasigroup over $\F$. Then
  $Q$ is isotopic to a group if and only if $a=b$.
\end{lem}

\begin{proof}
  First suppose that $a=b$, so that $(Q,*)$ is defined by
  $x*y=(1-a)x+ay$ for all $x,y\in\F$. So $Q$ is isotopic to the
  additive group of $\F$.

  For the remainder of this proof, suppose that $a\ne b$.
  We use the well known quadrangle criterion (see e.g.~\cite{Eva18}) to
  show that $Q$ is not isotopic to any group.
  This criterion states that if $Q$ is isotopic to a group and
  $r_1,r_2,c_1,c_2,r_1',r_2',c_1',c_2'$
are any elements of $Q$ such that $r_1*c_1=r_1'*c_1'$, $r_1*c_2=r_1'*c_2'$,
and $r_2*c_1=r_2'*c_1'$, then it follows that $r_2*c_2=r_2'*c_2'$.
  For $|\F|=7$ we apply this criterion to the following quadrangles in
  $Q_{3,5}$:
  \[ \begin{array}{r|cc} *&0&1\\ \hline 0&0&3\\
  1&3&1 \end{array} \qquad \begin{array}{r|cc} *&6&5\\ \hline 2&0&3\\
  4&3&0 \end{array} \] It follows that $Q_{3,5}$ (and also,
  by \pref{f3}(iv), its isomorph $Q_{5,3}$) is not isotopic to a
  group.

  Similarly, for $|\F|=9$ the following quadrangles 
  \[
  \begin{array}{r|cc}
    *&1&2i\\
    \hline
    1&1&2\\
    1+i&2&1
  \end{array}
  \qquad
  \begin{array}{r|cc}
    *&2+i&2+2i\\
    \hline
    1+2i&1&2\\
    i&2&0
  \end{array}
  \]
  show that $Q_{1+i,1+2i}$ is not isotopic to any group (where
  $i=\sqrt{-1}$). This property is necessarily inherited by the
  opposite quasigroup $Q_{2i,i}$ and translate $Q_{2+i,2+2i}$,
  as well as by $Q_{1+2i,1+i}$, $Q_{i,2i}$ and $Q_{2+2i,2+i}$. 
  There are no other solutions to \eref{r2} for $|\F|\le9$.
  
  If $|\F|>9$, then let $u,v\in\F$ be as given by \lref{l:runnonres}.
  For such $u,v$ we have the following violation of the quadrangle criterion:
  \[
  \begin{array}{r|cc}
    *&u-bu&u-bu+1\\
    \hline
    -bu&0&b\\
    1-bu&1-b&1
  \end{array}
  \qquad
  \begin{array}{r|cc}
    *&v-bv&v-bv+1\\
    \hline
    -bv&0&b\\
    1-bv&1-a&1
  \end{array}\qedhere
  \]
\end{proof}

\begin{lem}\label{f5}
Let $Q = Q_{a,b}$ be a quadratic quasigroup over $\F$.
If $-1$ is a square, then $Q$ is commutative
if and only if $a=b =1/2$. If $-1$ is a nonsquare,
then $Q$ is commutative if and only if $a+b =1$. 
\end{lem}

\begin{proof}
  We use \pref{f3}(v). If $|\F|\equiv1\bmod4$ then $Q_{a,b}$ is commutative
  if and only if $a=1-a$ and $b=1-b$.
  If $|\F|\equiv3\bmod4$ then $Q_{a,b}$ is commutative
  if and only if $a=1-b$ and $b=1-a$. The result follows.
\end{proof}

\begin{lem}\label{f6}
Let $Q = Q_{a,b}$ be a quadratic quasigroup over $\F$.
The quasigroup $Q$ is entropic (or left distributive,
or right distributive) if and only if $a=b$.
\end{lem}

\begin{proof} 
If $a=b$, then $Q$ is entropic (this is a well known fact
that may be verified directly).
Idempotent entropic quasigroups are
left distributive since an idempotent entropic quasigroups
fulfils $x\cdot yz = xx \cdot yz = xy\cdot xz$.
By \pref{f3}(v) it thus suffices to assume
that $Q$ is left distributive, and show that then $a = b$.

Since the number of squares in $\F$ exceeds $|\F|/2$, the squares
cannot form a subquasigroup of $Q$. Hence,
there exist squares $x,y\in \F$ such that $x*y$ is a nonsquare.
By \eref{r1}, $0*(x*y) = b(x*y)$. If $a$ is a square, then
$(0*x) * (0*y) =ax * ay=a(x*y)$, by 
\pref{f3}(iii). Thus in this case the left distributivity 
clearly implies $a = b$. If $a$ is a nonsquare, then 
$b$ is a nonsquare too. If $a\ne b$, then
$0*(x*y)=b(x*y)\ne bx * by = (0*x)*(0*y)$ by \pref{f3}(vii).
\end{proof}

A quasigroup $(M,\circ)$ is said to be \emph{affine} if
$x\circ y=x+\vhi(y-x)$ for all $x,y\in M$, where $\vhi\in\aut(M)$ 
for some abelian group defined on $M$. An affine quasigroup
is isotopic to the abelian group $(M,+)$ since 
$x\circ y = (1-\vhi)(x) + \vhi(y)$. By \lref{l:grpbased},
$Q_{a,b}$ is isotopic to a group if and only if $a=b$.
If $a=b$, then $Q_{a,b}$ is affine. 
We will say that $Q_{a,b}$ is \emph{non-affine} if $a\ne b$.

\begin{lem}\label{fsemi}
Let $Q = Q_{a,b}$ be a quadratic quasigroup over $\F$.
The quasigroup $Q$ is semisymmetric if and only if 
$a^2-a+1=0$ and either $a=b$ or $a+b=1$. 
\end{lem}

\begin{proof}
  We use \pref{f3}(vi) and the fact that $Q_{a,b}$ is semisymmetric if and
  only if it equals its translate.
  First suppose that $\chi(a)\ne\chi(-1)$. Then $Q_{a,b}$ equals its translate
  if and only if $a-1=ab=b-1$, which is equivalent to $a=b=a^2+1$.
  Next suppose that $\chi(a)=\chi(-1)$. Then $Q_{a,b}$ equals its translate
  if and only if $a,b$ are both solutions to $x=(x-1)/x$. There are two
  possibilities. The first is that $a=b=a^2+1$. The second is that $a,b$ are
  the two distinct roots of $x^2-x+1=0$, in which case $a+b=1$.
  The result now follows from the observation that if
  $a^2-a+1=0$ and $a+b=1$ then $\chi(a)=\chi(b)=\chi(-a^2)=\chi(-1)$.
\end{proof}

\begin{lem}\label{f8}
Let $Q=Q_{a,b}$ be a quadratic quasigroup over $\F$.
If $\chr(\F) = 3$, then $Q$ is a Steiner quasigroup if and only
if $a=b=-1$. In such a case, $Q$ is induced by an affine \STS.
If $\chr(\F)\ne 3$, then $Q$ is a Steiner quasigroup 
if and only if $ab= 1 =a+b$ and $-1$ is a nonsquare. 
In such a case we have $a\ne b$,
\begin{equation}\label{ef1}
\begin{gathered}
  ab=a+b=a-a^2=b-b^2=-a^3=-b^3=1\text{ \ and \ }\\
  \chi(a)=\chi(b)=\chi(1-a)=\chi(1-b)=\chi(-1) = -1.
\end{gathered}
\end{equation}
If $Q_{a,b}$ is a non-affine Steiner quasigroup, then $Q_{c,d}$ is another
non-affine quadratic Steiner quasigroup over $\F$ if and only if 
$\{a,b\} = \{c,d\}$. In such a case, $Q_{a,b}\cong Q_{c,d}$.
\end{lem}

\begin{proof}
  Since total symmetry is the combination of semisymmetry with commutativity,
  we combine Lemmas \ref{fsemi} and \ref{f5}. Together they imply that the
  necessary and sufficient conditions for total symmetry are that at least
  one of
\begin{align}
 &a^2-a+1=0\text{ \ and \ }a=b=1/2,\text{ \ or}\label{e:conda}\\
 &a^2-a+1=0\text{ \ and \ }a+b=1\text{ \ and \ }\chi(-1)=-1\label{e:condb}
\end{align}
holds. Condition \eref{:conda} implies that $3/4=0$ so it can only be
achieved if $\chr(\F) = 3$. Moreover, if $\chr(\F) = 3$
then \eref{:conda} is equivalent to $a=b=-1$.

Condition \eref{:condb} implies that $a(1-b)-a=-1$ and hence
$ab=1$. Moreover, if $ab=1$ and $a+b=1$ then $a^2=a(1-b)=a-1$.
The characterisation of quadratic Steiner
quasigroups follows.

Assume \eref{:condb} holds. Then
$b-b^2=b(1-b)=(1-a)a=1$. So $a,b$ are both roots of $x^2-x+1=0$ and
hence also of $x^3+1=(x+1)(x^2-x+1)=0$. From $a^3=b^3=-1$ and $a+b=1$ we get
$\chi(a)=\chi(b)=\chi(-1)=\chi(1-b)=\chi(1-a)$.

The polynomial $x^2-x+1$ has at most two roots (these
roots coincide if and only if $\chr(\F) = 3$).
That explains why $\{a,b\} = \{c,d\}$ if $Q_{a,b}$ and $Q_{c,d}$ are Steiner
quadratic quasigroups on $\F$, with $a\ne b$ and $c\ne d$.
If $\{a,b\} = \{c,d\}$, then
$Q_{a,b}\cong Q_{c,d}$ by \pref{f3}(iv).
\end{proof}

For the next proof we define notation $\mul{x}$ by $\mul{x}=a$ if
$\chi(x)=1$ and $\mul{x}=b$ if $\chi(x)=-1$. Note that $\mul{a}=\mul{b}$ and
$\mul{1-a}=\mul{1-b}$ for any quadratic quasigroup $Q_{a,b}$.

\begin{lem}\label{f9}
  The quadratic quasigroup $Q=Q_{a,b}$ is flexible if and only if at
  least one of the following conditions holds:
  \begin{itemize}
  \item[(i)] $a=b$,
  \item[(ii)] $\chi(a) = \chi(1-a)=1$, or
  \item[(iii)] $a+b=1$ and $|\F|\equiv3\bmod4$.
  \end{itemize}
\end{lem}

\begin{proof}
First note that $x*(x*x)=x*x=(x*x)*x$ for all $x\in\F$, by idempotence.
Thus consider distinct $x,y\in \F$ and let $z=x-y$. Then
\begin{equation*}
x*(y*x)=x*(y+\mul{z}z)=x+\mul{z(a-1)}z(\mul{z}-1),
\end{equation*}
and
\begin{equation*}
  (x*y)*x=(x-\mul{-z}z)*x=x-\mul{-z}z+\mul{az}\mul{-z}z
  =x+\mul{-z}z(\mul{az}-1).
\end{equation*}
It follows that $Q$ is flexible if and only if
\begin{equation}\label{eflex}
  \mul{z(a-1)}(\mul{z}-1)=\mul{-z}(\mul{az}-1)
\end{equation}
for all $z$.
Both sides of \eref{flex} are members of $\Phi=\{a(a-1),a(b-1),b(a-1),b(b-1)\}$.
If $a=b$ then $|\Phi|=1$, so \eref{flex} is
automatically satisfied. Henceforth, we assume $a\ne b$. In this case,
$|\Phi|=4$ unless $a(a-1)=b(b-1)$, which requires $a=1-b$.

Suppose for the moment that $\chi(z)=1$, meaning that
\eref{flex} is equivalent to $\mul{a-1}(a-1)=\mul{-1}(\mul{a}-1)$.
From the above observations, this condition can only be satisfied if
\begin{align}
  &\text{$\mul{a-1}=\mul{-1}$ and $\mul{a}=a$, or}\label{econd1}\\
  &\text{$a=1-b$ and $\mul{a-1}=a$ and $\mul{-1}=\mul{a}=b$.}\label{econd2}
\end{align}
Now \eref{cond1} is equivalent to condition (ii), whereas
\eref{cond2} implies (iii). It follows that for $Q_{a,b}$ to be flexible
it is necessary that (i), (ii) or (iii) holds.

To check sufficiency, we first note that if (iii) is true then
$\chi(a)=\chi(1-a)$ so either (ii) or \eref{cond2} holds. Hence if
(ii) or (iii) holds then at least one of \eref{cond1} or \eref{cond2}
holds.  Moreover, \eref{cond1} implies that $\mul{z(a-1)}=\mul{-z}$
and $\mul{z}=\mul{az}$, whereas \eref{cond2} implies that
$\mul{-z}=1-\mul{z}$ and $\mul{az}=\mul{(1-a)z}=1-\mul{z(a-1)}$.
In either case, \eref{flex} holds for all $z$.
\end{proof}

We note in passing that idempotent entropic quasigroups are always flexible
because $x(yx) = (xx)(yx) = (xy)(xx) = (xy)x$. Commutative quasigroups
are flexible too, because $x(yx) = (yx)x = (xy)x$.  These two facts
support the observation that the conditions encountered in
Lemmas~\ref{f6} and~\ref{f5} are incorporated in \lref{f9}.

\section{Subquasigroups and affine automorphisms}\label{ss}

If $Q$ is an idempotent quasigroup, then the term
\emph{trivial subquasigroup} refers to a quasigroup
consisting of at most one element. A \emph{minimal subquasigroup}
is a nontrivial subquasigroup in which all proper subquasigroups 
are trivial. A \emph{$2$-generated subquasigroup}
is a subquasigroup $S$ for which there exists $A\subseteq S$
such that $|A|\le 2$ and $S$ is the smallest subquasigroup of $Q$
that contains $A$. 
Any $2$-element subset of a minimal subquasigroup $S$ generates
$S$. However, there may exist nontrivial $2$-generated 
subquasigroups that are not minimal.

The field $\F$ is a vector space over its prime field. Saying
that $U\subseteq \F$ is a \emph{subspace} (of $\F$) means that it is 
a subspace of that vector space.

The purpose of this section is to show that if $Q_{a,b}$ is not a
Steiner quasigroup, then each minimal subquasigroup of $Q_{a,b}$ is
formed by a coset of a subspace of $\F$.  In particular, if a minimal
subquasigroup contains zero then it is closed under addition.

We start by two auxiliary observations that concern 
Frobenius groups.

\begin{lem}\label{f1}
Let $G$ be a Frobenius group that acts naturally on a finite set
$\Omega$.  If $G_\omega$ contains at most two nontrivial orbits for
some $\omega \in \Omega$, then the Frobenius kernel of $G$ is an
elementary abelian group.
\end{lem}

\begin{proof} Put $n=|\Omega|$. By the assumptions there exists
an orbit $\Gamma$ of $G_\omega$ such that $|\Gamma|+1 > n/2$. This
implies that $G$ is primitive. Now, elements of $\Omega$ may be
identified with elements of a group $N$ (which is isomorphic
to the Frobenius kernel), and $G_\omega$ may be identified 
with a subgroup of $\aut(N)$. Since there are at most two 
nontrivial orbits of $G_\omega$, 
there are at most two integers that occur as an order of a nontrivial
element of $N$. One of these integers has to be a prime, and the
other (if it exists) is either a square of this prime, or another 
prime.
This means that $N$ is solvable. A finite solvable group is
either elementary abelian, or it contains a nontrivial proper 
characteristic
subgroup. However, such a subgroup yields a 
block $B\subseteq \Omega$.
That is not possible since $G$ is primitive.
\end{proof}

\begin{lem}\label{f2}
Let $G$ be a Frobenius group that acts naturally on a finite
set $\Omega$. Suppose that the
Frobenius complement of $G$ is abelian.
If $\alpha,\beta\in \Omega$, $\alpha \ne \beta$,
$\vhi\in G_\alpha$, $\psi\in G_\beta$ and $\id_\Omega
\notin \{\vhi,\psi\}$, then $[\vhi,\vhi^{\psi}]$ is
fixed point free.
\end{lem}

\begin{proof} Denote by $N$ the Frobenius kernel of $G$ and recall that
non-identity elements of $N$ act without fixed points.
Since $G/N$ is abelian, we must have $[g,h]\in N$
for any $g,h\in G$. To prove that $[g,h]$ is fixed point
free it therefore suffices to find any $\omega\in \Omega$
such that $gh(\omega) \ne hg(\omega)$.
We take $\omega=\psi\m(\alpha)$ and observe that
\[\vhi\vhi^\psi(\omega) = \vhi\omega
\text{ and } \vhi^\psi\vhi(\omega) = \psi\m\vhi\psi(\vhi\omega).\]
Suppose that $\psi(\vhi\omega)=\vhi\psi(\vhi\omega)$. By definition,
$\alpha$ is the only fixed point of $\vhi$ and $\beta$ is the only
fixed point of $\psi$ so we can deduce in turn that $\psi(\vhi\omega)=\alpha$,
$\vhi\omega =\omega$, $\omega = \alpha$, $\psi(\alpha)=\alpha$
and thus $\alpha = \beta$. This contradiction proves that $\vhi\vhi^\psi(\omega)\ne\vhi^\psi\vhi(\omega)$, from which the result follows.
\end{proof}

A mapping $x\mapsto ux + v$, with $u \in \F^*$ and $v\in\F$, is
said to be an \emph{affine permutation} of $\F$. 
Affine permutations form a sharply $2$-transitive group.
Those with $u$ a square form a subgroup of $\aut(Q)$,
for any quadratic quasigroup $Q=Q_{a,b}$ over $\F$,
by \pref{f3}(ii),(iii). Hence we have:

\begin{prop}\label{f4}
Suppose that $Q=Q_{a,b}$ is a quadratic quasigroup over $\F$.
Let $s,t,u,v\in \F$ be such that 
$\chi(s-t) = \chi(u-v)$. Then there exists an affine automorphism
$\alpha\in \aut(Q)$ such that
$\alpha(s) = u$ and $\alpha(t) = v$.
\end{prop}

\begin{proof}
If $s=t$ then $u=v$, so we may use $x\mapsto x+u-s$ for $\alpha$.
So assume $t\ne s$, meaning that $u\ne v$ as well.
Put $\lambda = (u-v)/(s-t)$ and $\mu = (vs-ut)/(s-t)$. The mapping
$x \mapsto \lambda x + \mu$ is an automorphism of $Q_{a,b}$ that 
sends $s$ to $u$ and $t$ to $v$. 
\end{proof} 

\begin{prop}\label{f7}
Suppose that $Q = Q_{a,b}$. If $\chr(\F) = 3$, then
each minimal subquasigroup of $Q$ is a 
coset of a subspace of $\F$. If $\chr(\F)\ne 3$,
then each minimal subquasigroup of $Q$ is either a
coset of a subspace of $\F$, or \eref{f1} holds.
\end{prop}

\begin{proof} Let $U$ be a minimal subquasigroup
of $Q$. Set $m = |U|$. For $\eps\in \{-1,1\}$ denote by 
$\sigma(\eps)$ the number of
$(u,v)\in U\times U$ such that $\chi(v-u) = \eps$. 
Since $\sigma(1)+\sigma(-1)=m(m-1)$ there exists 
$\eps\in \{-1,1\}$ such that $\sigma(\eps)\ge \binom
m2$. Fix such $\eps$, and fix also $s,t \in U$ such that 
$\chi(t-s)=\eps$. 

Denote by $A$ the permutation group on $U$ that is induced
by affine automorphisms acting on $U$.
Since $s$ and $t$ generate $U$, we see that 
$\psi(U) = U$ whenever $\psi \in \aut(Q)$ is such
that $\psi(s),\psi(t)\in U$. By \pref{f4},
for each $(u,v) \in U\times U$
with $\chi(v-u)=\eps$, there exists $\psi \in A$ such
that $\psi(s) = u$ and $\psi(t) = v$.
Therefore $|A|\ge \binom m2$. 

Denote by $f$ the number of fixed point free elements of $A$.
The aggregate number of fixed points of elements of $A$ is equal to
$m+|A| - 1 -f$. By Burnside's lemma this is equal to 
$|A|k$, where $k$ is the number of orbits of $A$.
If $k\ge 2$, then 
\[m-1-f = |A|(k - 1) \ge|A|\ge\dfrac12m(m-1).\]
It follows that $2(-1-f) \ge m(m-3)$ and hence $m<3$.
However, this is impossible since there
is no idempotent quasigroup of order two and $|U|>1$. 
Hence $k = 1$; in other words $A$ is transitive.
As each nontrivial element of $A$ fixes at most one element,
it follows that $A$ is either a regular group or a Frobenius group.
Since $|A| \ge m(m-1)/2$, the former alternative may take
place if and only if $m =3$. We distinguish two cases.

First, suppose that $|A|=3$. This means that $A$ is cyclic, and $|U|=m=3$
as well. 
Let $A$ be generated by an affine permutation
$x\mapsto \lambda x +\mu$. If $\lambda = 1$, then $U$
is a coset of a subspace and $\chr(\F)=3$. 
Suppose $\lambda \ne 1$. Then 
$\lambda^3 = 1$, so $\lambda^2 + \lambda + 1=0$.
Our intention is to show that $ab = 1 = a+b$.
By \pref{f3}(ii), it may be assumed that $0 \in U$ since 
automorphisms map minimal subquasigroups to minimal subquasigroups.
Suppose that $u \in U\setminus\{0\}$. Then $u=\mu$ and
$U=\{0,u,(\lambda+1)u\}$. Since $U$ is idempotent,
$u*0=0*u = (\lambda+1)u = -\lambda^2 u$. 
Suppose that $-1$ is a square. Then
$-\lambda^2u=u*0=u-0*u=(1+\lambda^2)u=-\lambda u$,
resulting in $\lambda = 1$.
Hence $-1$ is a nonsquare.
If $u$ is a square, then $0*u=au$ so $a=-\lambda^2$ and
$u=0*(-\lambda^2 u) = -\lambda^2 bu = abu$. Alternatively,
if $u$ is a nonsquare, then $0*u=bu$ so $b=-\lambda^2$
and $u=0*(-\lambda^2 u) = -\lambda^2 au = abu$. Therefore $ab = 1$
in every case. Since at least one of $a^3$ and $b^3$ is equal
to $(-\lambda^2)^3 = - 1$, it must be that $a^3 = -1 =b^3$.

Thus $\chi(a) = \chi(b) = \chi(-1) = -1$ and
$0 = a^3+1=(a+1)(a^2-a+1)$. If $a = -1$, then $b=-1$
and $-u = 0*u=u * 0 = u -(-u) = 2u$. In such a case $0 = 3$,
$U = \{0,u,-u\}$ is a subspace of $\F$, 
and $a+b = - 2=1$. Assume $a\ne -1$. Then $b\ne -1$,
$a^2-a=-1=b^2-b$ and $a+b = a+a\m = (a^2+1)/a=1$. 
Hence \eref{f1} holds, by \lref{f8}.

Let us now turn to the case $|A|>3$. The group $A$ is a Frobenius
group on an $m$-element set $U$ with complements of 
order $\ge (m-1)/2$. The kernel of
$A$ is elementary abelian, by \lref{f1}. To finish it suffices to
prove that all elements of the kernel of $A$ are induced by
translations $x\mapsto x+\mu$. Assume the contrary.  This means that
the kernel of $A$ contains a permutation that is a restriction of an affine
automorphism $\psi\colon x\mapsto \lambda x + \mu$, where $\lambda\ne1$.
Choose an affine automorphism $\vhi$ that induces a nontrivial
permutation belonging to a complement of $A$. Both $\vhi$ and $\psi$
are elements of the Frobenius group formed by affine automorphisms of
$Q$, and they belong to different Frobenius complements of the latter
group. Hence
$[\psi,\psi^\vhi]$ is a fixed point free permutation of $Q$, by
\lref{f2}. Denote by $\tilde \psi$ the restriction of $\psi$ to $A$. Both
$\tilde \psi$ and $\tilde \psi^\vhi$ belong to the kernel of $A$. This
kernel is abelian.  Hence $[\psi,\psi^\vhi]$ fixes each point of $U$,
a contradiction.
\end{proof}

\section{Existence and nonexistence of isomorphisms}\label{e}

The purpose of this section is to prove Theorems~\ref{r1} 
and~\ref{r2} for $2$-generated quadratic quasigroups, and
to discuss consequences of the existence
of an additive automorphism $\vhi\in \aut(Q)$ that maps
a square to a nonsquare. One of the tools will
be the following obvious fact:

\begin{lem}\label{e1} 
Let $x\in \F^*$. The least subfield containing $x$ coincides
with the set of all sums $x^{i_1}+\cdots + x^{i_k}$,
where $i_j$ is a nonnegative integer for $1\le j \le k$ and $k \ge 0$.
\end{lem}

The following easy facts may be deduced, e.g., from 
results of Perron \cite{per}. (To avoid a misunderstanding
let it be mentioned that while Perron's paper is formulated 
for prime fields only, the proofs of the paper carry without a change to 
any finite field of odd order.)

\begin{lem}\label{e2}
Assume $|\F|\ge 5$. Each square may be obtained as 
a sum of two nonsquares, and each nonsquare may be
obtained as a sum of two squares. 
For each $a\in \F$ there exist squares $x,y\in \F$
and nonsquares $x',y'\in \F$ such that $a+x$ and $a+x'$ are squares,
while $a+y$ and $a+y'$ are nonsquares.
\end{lem}

Note that if $Q_{a,b}$ is a quadratic quasigroup over
$\F_3$, then $a=b=-1$. In particular, if $|\F|<5$ then $Q_{a,b}$
has to be affine.

A permutation $\psi$ of $\F$ is said to be \emph{additive}
if $\psi(x+y) = \psi(x)+\psi(y)$ for all $x,y \in \F$.

\begin{lem}\label{e3}
Let $Q=Q_{a,b}$ be a non-affine quadratic quasigroup over $\F$, 
and let $\vhi \in \aut(Q)$ be additive. Suppose that
there exists $f\in \F$ such that $\vhi(a^i) = b^if$
for each integer $i$. Then $f$ is a nonsquare.

Furthermore, denote by $U$ the least subfield
of $\F$ that contains the subquasigroup
generated by $0$ and $1$. Then 
$U$ is a subquasigroup of $Q$ containing $a$ and $b$, and 
$b=a^\gamma$ for some $\gamma >1$ such that $\gamma^2 \bigm| |\F|$
and $\vhi(u) =u^\gamma f$ for each $u\in U$.

If $U$ contains an element
that is a nonsquare in $\F$, then $(U,*) = (\F_{\gamma^2},*_a)$.
\end{lem}

\begin{proof}
Since $0*1 = a$, we have $0*f=\vhi(0)*\vhi(1) = \vhi(a) = bf$.
Therefore $f$ is a nonsquare. 

Denote by $U_a$ the subfield of $\F$ generated
by $a$, and by $U_b$ the subfield of $\F$ generated by $b$. 
The field $U_a$ coincides with the set of all sums $a^{i_1}+\dots +
a^{i_k}$, $k\ge 0$, by \lref{e1}. 
The assumptions of the statement imply that
\begin{equation}\label{ee1}
\vhi (a^{i_1}+\dots +a^{i_k} )= (b^{i_1}+\dots +b^{i_k}) f.
\end{equation}
This means that the bijection $\vhi$ maps $U_a$ to $U_bf$.
Therefore there exists a bijection $\alpha\colon U_a\to U_b$ 
such that $\vhi$ sends $u\in U_a$ to $\alpha(u)f$. 
The form of $\alpha$ follows from \eref{e1},
and this form stipulates that $\alpha$ is an isomorphism
of fields $U_a \cong U_b$. Both of them are subfields
of $\F$. Since they are of the same order, they have to
coincide. 

Note that any subfield that contains both $a$ and $b$ forms
a subquasigroup of $Q$, by \eref{r1}. 
Hence $U_a$ is a subquasigroup that contains both $0$ and $1$.
Therefore $U_a \supseteq U$. Since $0*1 =a \in U$, we have
$U\supseteq U_a$ as well, so $U=U_a$. 

The field automorphism $\alpha \in \aut(U)$ extends to an automorphism
of $\F$ that sends each $x\in \F$ to $x^\gamma$, where $\gamma$
divides $|U|$, which in turn divides $|\F|$.

Every element of a subfield $\K$ of $\F$ is a square in $\F$
if and only if $|\F:\K|$ is an even number.
If $|\F:U|$ is even, then 
$\gamma^2$ divides $|\F|$. 

Suppose that $|\F:U|$ is odd. Then there exists $u\in U$ 
such that $u$ is a nonsquare in $\F$.
Since $0*u = bu=a^\gamma u$ and $\vhi$ is an automorphism,
$\vhi(a^\gamma u)=\vhi(0) *\vhi(u) = 0*u^\gamma f = au^\gamma f$.
Since $a^\gamma u \in U$,
we also have $\vhi(a^\gamma u) = a^{\gamma^2} u^\gamma f$.
Therefore $a^{\gamma^2} = a$. Since $U=U_a$ is the least
subfield containing $a$, we see that $x^{\gamma^2} = x$ for
every $x\in U$, by \lref{e1}. Also $a\ne b=a^\gamma$ so $\gamma\ne1$ and
$U = \F_{\gamma^2}$. The rest follows from \pref{r3}.
\end{proof}

\begin{lem}\label{e4}
Let $Q = Q_{a,b}$, be a quadratic quasigroup over $\F$, 
with $a$ and $b$ distinct nonsquares.
Let $\vhi \in \aut(Q)$ be additive and let $f=\vhi(1)$
be a nonsquare. Then there exists $\gamma>1$ such that $b=a^\gamma$,
$\gamma^2$ divides $|\F|$ and $\F_{\gamma^2}$ carries a subquasigroup
of $Q$ that coincides with $(\F_{\gamma^2},*_a)$.
\end{lem}

\begin{proof} 
The first step is to prove for each $i\ge 0$ that 
\begin{align}
\vhi(a^ib^i) &= a^ib^i f, \text{ and} \label{ee3}\\
\vhi(a^{i+1}b^{i})&=a^{i}b^{i+1} f.\label{ee2}
\end{align}
Now $\vhi(a)=\vhi(0*1)=\vhi(0)*\vhi(1)=0*f=bf$ so
\eref{e2} holds for $i=0$. Also \eref{e3} holds for $i=0$ by the definition
of $f$. Note that $a^ib^i$ is a square and $a^ib^{i-1}$ is a nonsquare.
So by induction on $i$ we find that
\begin{gather*}
  \vhi(a^{i}b^{i})=\vhi(0*a^ib^{i-1})=\vhi(0)*\vhi(a^ib^{i-1})=0*a^{i-1}b^if=a^{i}b^{i}f\text{ and} \\
  \vhi(a^{i+1}b^{i})=\vhi(0*a^ib^i)=\vhi(0)*\vhi(a^ib^i)=0*a^ib^if=a^{i}b^{i+1}f,
\end{gather*}
completing the proof of \eref{e3} and \eref{e2}.

Denote by $C$ the subfield of $\F$ generated by $ab$.
By \lref{e1}, each element of $C$ may be expressed as a sum
$(ab)^{i_1}+\dots + (ab)^{i_k}$. 
Hence $\vhi(c) = cf$ for each $c\in C$, by \eref{e3}.

Let us now assume that there exists $\nsq\in C$ such that
$\chi(\nsq) = -1$. In such a case the nonsquares of $C$ coincide with
the nonsquares of $\F$ contained in $C$, and, as we shall prove,
whenever
$x,y \in \F$ are such that $\chi(x) = \chi(y)$ and
that $\vhi(x) = yf$, then
\begin{equation}\label{ee4}
(\forall c\in C\colon\ \vhi(xc) = ycf) \quad
\Rightarrow \quad (\forall c\in C\colon\ \vhi(axc) = bycf).
\end{equation}
Let us assume that the hypothesis of \eref{e4} is true.
Our aim is to show that then $\vhi(axc) = bycf$ for
each $c\in C$. The first step is a choice
of $x',y'\in \F$. If $x$ is a square, put
$(x',y')=(x,y)$. If $x$ is a nonsquare, put
$(x',y') = (x\nsq,y\nsq)$, where $\nsq$ is as above.
By the hypothesis of \eref{e4}, $\vhi(x's) = y'sf$
for each square $s\in C$. Since $x's$ and $y's$
are squares, $\vhi(ax's)=\vhi(0*x's)=\vhi(0)*\vhi(x's)=0*y'sf=by'sf$ for 
each square $s\in C$. If $t\in C$ is another square,
then $\vhi(ax'(s+t))=by'(s+t)f$ since $\vhi$ is additive.
Therefore $\vhi(ax'c) = by'cf$ for each $c\in C$, by
\lref{e2}. That finishes the proof of \eref{e4}
since $\nsq \in C$.

Assuming the existence of $\nsq$, \eref{e4} implies
$\vhi(a^ic) = b^i c f$, for each $i \ge 0$
and $c\in C$. Indeed, since we have proved that $\vhi(c) = cf$ 
for each $c\in C$, the
equality holds for $i=0$. The induction step follows
from \eref{e4}, by setting $(x,y) = (a^i,b^i)$.

We have shown that if $C$ carries a nonsquare in $\F$,
then $\vhi(a^i) =b^if$ for every $i \ge 0$. That allows
us to draw the needed conclusions from \lref{e3}. 
So, for the rest of the proof we may assume
that each element of $C$ is a square in $\F$. In particular,
$\chi(-1) = 1$. 

Now $\vhi(a)=\vhi(0*1) =\vhi(0)*\vhi(1) = 0*f =bf$
and $\chi(a-1)=\chi(b-1)$, by \eref{r2}. Thus the claim
\begin{equation}\label{ee5}
\vhi(a^i)=b^if \text{\, and \,} \chi(a^i-1)=\chi(b^i-1)
\end{equation} 
holds for $i=1$. Let us now show that if $\vhi(a^j)=b^jf$
holds for each positive $j\le i$, then $\chi(a^i-1)=\chi(b^i-1)$.
Since $\chi(a-1)=\chi(b-1)$, it suffices to
show that $\chi(A)=\chi(B)$, where
$A=\sum_{j=0}^{i-1}a^j$ and $B=\sum_{j=0}^{i-1}b^j$.
Note that $\vhi(A) = Bf$. It follows that $A=0$ if and only if $B=0$,
so we may assume that $A\ne0$ and $B\ne0$.
Suppose that $A$ is a nonzero square, so that $0*A=aA$.
Since $\vhi$ is an additive automorphism, we
must have $0*fB=\vhi(aA)=fbB$. However, that
is possible if and only if $B$ is a square. Conversely,
if $B$ is a square, then
$\vhi(0*A)=0*fB=fbB=\vhi(aA)$, 
implying that $0*A = aA$. Hence $A$ is 
a square in $\F$ if and only if $B$ is a square in $\F$,
and the same holds for $a^i-1$ and $b^i-1$.

Thus \eref{e5} holds for all $i\ge 1$ if we can prove
that its validity for a given $i\ge 1$ implies $\vhi(a^{i+1}) = b^{i+1}f$.
For even $i$ we need only observe that
$\vhi(a^{i+1})=\vhi(0*a^i)=\vhi(0)*\vhi(a^i)=0*b^if=b^{i+1}f$.
So we may suppose that $i$ is odd.

Assume first that $a^i-1$ is a square.
Then $$\vhi(a^{i+1}-a)=\vhi(0*(a^i-1))=\vhi(0)*\vhi(a^i-1)=0*(b^i-1)f = b^{i+1}f - bf.$$
Since $\vhi(-a) =-bf$ and $\vhi$ is additive, we have $\vhi(a^{i+1})=b^{i+1}f$. 

Suppose now that $a^i-1$ is a nonsquare. Then
$b^i-1$ is a nonsquare too, and $(b^i-1)(-a^i) = a^i - (ab)^i$
is a square. By the inductive assumption and \eref{e3} we see that
$$\vhi(a^i - (ab)^i) = (b^i -(ab)^i)f=(a^i-1)(-b^i)f$$
is a nonsquare. Hence, by \eref{e2},
\begin{align*}
  \vhi(a^{i+1})-a^ib^{i+1}f&=\vhi(a^{i+1} - a^{i+1}b^i)=\vhi(0*(a^i-(ab)^i)) = 
  0*(b^i -(ab)^i)f\\
  &=b^{i+1}f - a^ib^{i+1}f.
\end{align*}
Therefore $\vhi(a^{i+1})=b^{i+1}f$, completing the proof of \eref{e5}.
We see that \lref{e3} may be applied in this case too.
\end{proof}

\begin{prop}\label{e5}
Let $Q = Q_{a,b}$ be a non-affine quadratic quasigroup
over $\F$. Suppose that there exists 
$\vhi \in \aut(Q)$ such that $\vhi$ is additive and
$\vhi(1)$ is a nonsquare. Then $b = a^\gamma$, where $\gamma^2$
divides $|\F|$ and $\gamma>1$. The subfield $U$ generated
by $a$ is equal to $\F_{\gamma^2}$ and forms a subquasigroup
of $Q$ such that $(U,*) = (\F_{\gamma^2},*_a)$. 
\end{prop}

\begin{proof} If $a$ and $b$ are nonsquares, then the 
result follows from \lref{e4}, so we assume that $a$ and 
$b$ are squares. It is then easy to show by induction that
$\vhi(a^i) = b^i f$ for every $i \ge 1$, where $f=\vhi(1)$.
Indeed $\vhi(a) = bf$ since $0*1 =a $, while $\vhi(a^{i+1}) 
= \vhi(0*a^i) = 0 * b^if = b^{i+1}f$ yields the induction step.
By \lref{e3}, $U$ is a subquasigroup, $b =a^\gamma$
for $\gamma >1$ such that $\gamma^2 \bigm| |\F|$, and the statement
is true if $U$ contains an element that is a nonsquare in $\F$.
Thus, for the rest of the proof it will be assumed that each
element of $U$ is a square in $\F$. Let us also stipulate that
$\gamma > 1$ is the least possible. 

The next step is to show that if $\chi(\vhi(f))=1$, then 
$\vhi(u^\gamma f) = u \vhi(f)$, while for $\chi(\vhi(f))=-1$
we have $\vhi(uf) = u \vhi(f)$, for each $u \in U$.

We first assume $\chi(\vhi(f)) = 1$ and employ induction
to prove that $\vhi(b^if) = a^i\vhi(f)$ for each $i\ge 0$.
The case $i=0$ is trivial, and
$\vhi(b^{i+1}f) = \vhi(0*b^if) = 0*a^i\vhi(f) = a^{i+1}\vhi(f)$ completes
the induction.
Each $u\in U$ may be expressed as $\sum a^{i_j}$, by
\lref{e1}. In such a case $u^\gamma = \sum b^{i_j}$
and $\vhi(u^\gamma f) = u\vhi(f)$, by the additivity of~$\vhi$.

Assume now that $\chi(\vhi(f)) = -1$. Since $b$ is the
image of $a$ under the field automorphism $x\mapsto x^\gamma$,
we have $a = b^{\gamma'}$, where $x\mapsto x^{\gamma'}$ is the inverse
automorphism.
Thus each $u\in U$ may be expressed as
$\sum b^{i_j}$, by \lref{e1}.
Also, $\vhi(b^if) = b^i\vhi(f)$ for each $i\ge 0$ by induction,
since $\vhi(b^{i+1}f) = \vhi(0*b^if) =0 * b^i\vhi(f) = b^{i+1}\vhi(f)$.
Therefore $\vhi(uf) = u\vhi(f)$ for each $u \in U$, as claimed.

The automorphism $\vhi$ may be replaced by its composition
with the affine isomorphism $x\mapsto cx$, where $c$ is a square.
Hence $f$ may be equal to any nonsquare in $\F$. 
By \lref{e2}, we may choose $f$ in such a way that $1+f$ is a square.
Then
\begin{equation}\label{ee6}
  a^\gamma f + \vhi(af)=\vhi(a+af)=\vhi(0*(1+f))=0*\vhi(1+f)
  \in \{af + a\vhi(f), 
  \, a^\gamma f + a^\gamma \vhi(f)\},
\end{equation}
by \eref{r1}.
If $\vhi(f)$ is a nonsquare, then $\vhi(af) = a\vhi(f)$,
by the results above.
This contradicts \eref{e6} since $a^\gamma \ne a$. Hence
$\vhi(f)$ is a square and $\vhi(af) = a^{1/\gamma} \vhi(f)$.
Suppose first that 
\[ a^\gamma f + a^{1/\gamma}\vhi(f) = af + a\vhi(f)\]
and choose $d\in \F$ such that 
\[ d^\gamma = \frac{a^\gamma - a}{a-a^{1/\gamma}}.\]
Then $\vhi(f) = d^\gamma f$. Because $d\in U$, we also
have $\vhi(d) = d^\gamma f$, by \lref{e3}. 
Thus $d =f $. This is a contradiction
since $d$ is a square and $f$ is a nonsquare. Therefore
$a^{1/\gamma}\vhi(f) =a^\gamma \vhi(f)$, and $a^{\gamma^2} =a$.
Hence $|U|$ divides $\gamma^2$ and admits a nontrivial involutory 
automorphism $x\mapsto x^\gamma$. Since $\gamma$ has been chosen
to be the least possible, $\gamma$ is a proper divisor of $|U|$.
Thus $|U| = \gamma^2$.
\end{proof}

\begin{lem}\label{e6}
Let $Q_{a,b}$ be a $2$-generated quadratic quasigroup over 
$\F$. Then at least one of the quasigroups $Q_{a,b}$
and $Q_{b,a}$ is generated by $\{0,1\}$.  
\end{lem}
\begin{proof} Put $Q=Q_{a,b}$. Since $\aut(Q)$ is transitive,
there exists $u\in Q$ such that $\{0,u\}$ generates $Q$.
If $u$ is a square, then $Q$ is generated by $\{0,1\}$
since $x\mapsto ux$ belongs to $\aut(Q)$. Assume that
$u$ is a nonsquare. Then $x\mapsto u\m x$ is an
isomorphism $Q_{a,b}\cong Q_{b,a}$, by
\pref{f3}(iv), and this isomorphism sends $\{0,u\}$
to $\{0,1\}$.
\end{proof}

\begin{thm}\label{e7}
Let $Q=Q_{a,b}$ be a $2$-generated quadratic quasigroup over    
$\F$. Then $G=\aut(Q)$ is $2$-transitive if and only
if $a=b$ or  $Q=(\F_{\gamma^2},*_a)$. 

In the former case
$G$ consists of all mappings $x\mapsto \lambda x+\mu$,
where $\lambda \in \F^*$ and $\mu \in \F$.
In the latter case $G$ consists of mappings 
$x\mapsto \lambda x+\mu$ and mappings $x\mapsto \lambda'
x^\gamma + \mu$, where $\lambda,\lambda',
\mu \in \F$, $\chi(\lambda)=1$ and $\chi(\lambda') = -1$.

If $G$ is not $2$-transitive, then it consists of
all mappings $x\mapsto \lambda x + \mu$,
where $\lambda,\mu \in \F$ and $\chi(\lambda) = 1$.
\end{thm}

\begin{proof} By \pref{f3}, $G$ contains all
mappings $x\mapsto \lambda x + \mu$, where $\lambda \in \F^*$
is a square. Suppose first that $G$ is $2$-transitive.
Then $G$ is sharply $2$-transitive since an automorphism
that fixes generators pointwise has to be the identity mapping.
The mappings $x\mapsto x+\mu$ thus form a normal
subgroup of $G$, and that makes each $\vhi \in G_0$
additive (where $G_0$ is the stabiliser of $0$ in $G$).
\pref{e5} hence confirms that
$G$ may be $2$-transitive only in the cases described
above. All mappings mentioned so far
are automorphisms of $Q$, and they form a $2$-transitive group.
No other automorphism of $Q$ may thus exist.

Let us now turn to the case when $G$ is not $2$-transitive.
Let $Q$ be generated by $\{0,u\}$. Then $\chi(u) = \chi(\vhi(u))$
for each $\vhi \in G_0$ since otherwise $G$ is $2$-transitive.
For each $\vhi \in G_0$ there thus exists a square $\lambda \in
\F^*$ such that $\vhi(u) = \lambda u$. Since $\vhi$ and 
$x\mapsto \lambda x$ agree on a set of generators, they
agree everywhere. Nothing else is needed.
\end{proof}

\begin{thm}\label{e8}
Let $Q=Q_{a,b}$ be a $2$-generated quadratic quasigroup over
$\F$. Then $Q \cong Q_{c,d}$ if and only if there exists
$\alpha\in \aut(\F)$ such that $\{c,d\} = \{\alpha(a),\alpha(b)\}$.
\end{thm}

\begin{proof} By \pref{f3} only the direct implication needs to 
be proved. Fix an isomorphism $\psi\colon Q\mapsto Q_{c,d}$ and put $G=\aut(Q)$.
The group $G$
is $2$-transitive if and only if $\aut(Q_{c,d})$ is 
$2$-transitive. If $a=b$, then $c=d$ since this is the only
$2$-transitive case in which $G_0$ is abelian, by \tref{e7}.
Hence \tref{e7} implies that $G$ is equal to $\aut(Q_{c,d})$ in all cases.
Therefore $\psi$ normalises $G$, and hence
$\psi$ also normalises the group of translations $x\mapsto x+\mu$.
Since $G$ is transitive, $\psi(0) = 0$ may be assumed. The
normalising property means that $\psi$ is additive and
normalises $G_0$.

By \lref{e6} it may be assumed that $Q$ is generated
by $0$ and $1$. Then $0$ and $\psi(1)$ generate $Q_{c,d}$.
After a possible switch of $c$ and $d$ we may thus assume
that $\psi(1) = 1$ as well, by \pref{f3} (iii), (iv).  

Denote by $\sigma_s$ the multiplication
$x\mapsto sx$, where $s\in \F^*$. If $\sigma_s\in G_0$, then
$\psi \sigma_s \psi\m = \sigma_t$ for some $t\in \F^*$ since
$\psi$ normalises $G_0$. Since $\psi(1) = 1$, we must have $t=\psi(s)$. 
Hence $\psi(sy)=\psi\sigma_s(y) = \sigma_{t}\psi(y)=\psi(s)\psi(y)$
for all $y\in\F$, whenever $\sigma_s \in G_0$. This shows
that $\psi\in \aut(\F)$ if $a=b$.

Suppose that $a\ne b$.
We shall show that $\psi\in \aut(\F)$ in this case too. 
Indeed, if $x\in \F$ is a nonsquare, then $x$ may be expressed
as $u+v$, where both $u$ and $v$ are squares, by \lref{e2}.
In such a case
$\psi(xy) = \psi(uy + vy) = (\psi(u)+\psi(v))\psi(y) = \psi(x)\psi(y)$
for all $y\in\F$.

To finish, note that
in $Q_{c,d}$ both $c=0*1 = \psi(0)*\psi(1) =\psi(a)$ and
$\psi(b)\psi(\nsq) = \psi(b\nsq) = \psi(0*\nsq) = 0*\psi(\nsq)= d\psi(\nsq)$
hold, where $\nsq$ is any nonsquare. Hence $c = \psi(a)$ and
$d=\psi(b)$.
\end{proof}

\section{Subfields and subquasigroups}\label{s}

In this section we examine the structure of minimal subquasigroups and
2-generated subquasigroups of quadratic quasigroups. Note that in
Steiner quasigroups every pair of elements generates a (minimal)
subquasigroup of order 3, by definition. As this case is trivial,
we may for convenience exclude certain Steiner quasigroups from
our discussions in this section.

Let us start with an easy general fact:

\begin{lem}\label{s1}
Let $Q$ be a finite quasigroup and let $\alpha$ be an automorphism
of $Q$. Suppose that $S$ is a subquasigroup
of $Q$ that is generated by a set $X$. Then $\alpha(S) = S$
if and only if $\alpha(X)\subseteq S$.
\end{lem}

\begin{proof} If $\alpha(S) = S$, then $\alpha(X)\subseteq S$.
Conversely, assume $\alpha(X)\subseteq S$ and denote by $S'$ the 
subquasigroup generated by $\alpha(X)$. We have
$\alpha(X)\subseteq S'\cap S$, so $S' \subseteq S$.
However, $\alpha$ is an isomorphism from $S$ to $S'$ so $|S'|=|S|$,
which means that $\alpha(S)=S' = S$. 
\end{proof} 

\begin{lem}\label{s2}
Let $Q=Q_{a,b}$ be a quadratic quasigroup over $\F$.
The set of all $\sum a^{i_k}b^{j_k}$, where $1\le k \le r$,
with $i_k \ge 0$, $j_k \ge 0$ and $r\ge 0$, coincides
with the least subfield of $\F$ that contains $\{a,b\}$.
This subfield is a subquasigroup of $Q$.
\end{lem}

\begin{proof} The set is closed under sums and products,
and hence a subfield. By Definition \eref{r1},
it is closed under $*$ as well.
\end{proof}

\begin{lem}\label{s3}
Let $Q$ be a quadratic quasigroup over $\F$ that is not
a Netto quasigroup. Let $S$ be a minimal subquasigroup of $Q$, with
$\{0,1\}\subseteq S$. Then $S$ is a subfield of $\F$ and $a\in S$. 
If $S$ contains a nonsquare, then $b\in S$, and $S$ is equal
to the least subfield of $\F$ that contains $a$ and $b$. If $S$ is composed
of squares only, then $S$ coincides with the least subfield
of $\F$ that contains $a$. 
\end{lem}

\begin{proof}
By \pref{f7} and \lref{f8}, $S$ is a subspace of $\F$.  Moreover, if
$Q$ is affine Steiner, then $S$ is equal to $\{0,1,-1\}$ and is a
subfield that contains $a=b=-1$. For the rest of the proof it thus may
be assumed that $Q$ is not a Steiner quasigroup.

Put $S_0 = \{\sum x_i:x_i\in S$ and $\chi(x_i) \ge 0 \}$ and note that
$S_0\subseteq S$. If $s\in S$ is a nonzero square,
then the automorphism $x\mapsto sx$ sends $0$ to $0$ and 
$1$ to $s\in S$. Thus $0$ and $s$ generate $S$, by \lref{s1}.
Hence $sS= S$ and $sS_0\subseteq S_0$. Since
$S_0$ is a subspace, $S_0 S_0\subseteq S_0$, and that implies
that $S_0$ is a subfield of $\F$. Furthermore, $S$ is a vector
space over $S_0$, since $x_iS \subseteq S$ yields $(\sum x_i)S 
\subseteq S$.

If $\{a,b\}\subseteq S_0$, then $S=S_0$, by \lref{s2}. Note that
$a=0*1$ always belongs to $S$. 

Let $a$ be a square. Then $a^i\in S_0$ for each $i\ge 0$ since
$0*a^i= a^{i+1}$.  Hence $S_0\supseteq S_a$ by \lref{e1}, where $S_a$ is the
least subfield of $\F$ that contains $a$. If $S_0$ consists only of
squares, then $S_a$ is a subquasigroup, and $S_a=S_0=S$.  Suppose that
$S_0$ contains a nonsquare, say $\nsq$.  To show that $b\in S_0$ it
suffices to show that $b\in S$, since $b$ is a square. Now, $b\nsq = 0
* \nsq \in S$. Hence $b=\nsq\m (\nsq b)\in S$, since $\nsq\m \in S_0$
and $S_0$ acts on $S$.

For the remainder of the proof,
let $a$ and $b$ be nonsquares. If $S_0$ contains a
nonsquare, say $\nsq$, then $b\nsq = 0*\nsq \in S_0$.
Therefore $b = \nsq\m(\nsq b) \in S_0$. Since $a\in S$,
$0*a = ab \in S_0$, and hence $a = b\m (ab) \in S_0$.
Thus if $S_0$ contains a nonsquare, then $\{a,b\}\subseteq S_0=S$,
and $S$ contains the subfield generated
by $a$ and $b$. In such a case the subfield is equal to $S$,
since the subfield is a subquasigroup containing $a$ and $b$, by \lref{s2}.

What remains is the case in which $S_0$ contains only squares, 
i.e.~the squares of $S$ form a subfield. 
We shall show that this may be always brought
to a contradiction. Consider distinct $s,t\in S_0$.
If $s*t = s+a(t-s)$ is a square, i.e.~$s*t\in S_0$,
then $a(t-s) = s*t -s$ is a square too, a contradiction.
Hence $s*t \in N_1 = S\setminus S_0$. If $s\in S_0\setminus\{0\}$ 
then $0*s = as \in N_1$. If $n\in N_1$,
then $0*n = bn \in S_0$. Therefore $|S_0|=|N_1|+1$.
Consider now the multiplication table of $(S,*)$. Note
that $S$ is a disjoint union of $S_0$ and $N_1$. The subtable
$S_0\times S_0$ has elements of $S_0$ on the diagonal,
and the rest is occupied by elements of $N_1$. Therefore
all entries in subtables $S_0\times N_1$ and $N_1\times S_0$
belong to $S_0$. Therefore all entries in $N_1\times N_1$
are from $N_1$, and that makes $N_1$ a subquasigroup.
Since $S$ is a minimal subquasigroup and $N_1\ne S$,
the only possibility is that $N_1=\{a\}$ and $S_0 = \{0,1\}$.
Since $S$ is a subspace, $\chr(\F) = 3$, $a=-1$ and
$ab = 0*a = 0*-1 = 1$. That implies $b=-1$. 
By \lref{f8} this means that $Q$ is a Steiner quasigroup,
contrary to our assumptions.
\end{proof}

\begin{thm}\label{s4} 
Let $Q =Q_{a,b}$ be a quadratic quasigroup over $\F$
that is not a Netto quasigroup. Let $\K$, $\K_0$ and $\K_1$ be the
subfields generated by $\{a,b\}$, $\{a\}$ and $\{b\}$, respectively.
Suppose that each subquasigroup
of $Q$ that is generated by two distinct elements is minimal.
There are two possibilities:
\begin{enumerate}
\item[(i)] $\K$ contains an element that is a nonsquare in $\F$,
and $\K = \K_0=\K_1$. The minimal subquasigroups of $Q$
are exactly  the
sets $\lambda \K + \mu$, where $\lambda \in \F^*$ and
$\mu \in \F$. 
\item[(ii)] All elements of $\K_i$, $i\in \{0,1\}$,
are squares in $\F$. If $\nsq \in \F$ is a nonsquare, then
the minimal
subquasigroups of $Q$ are exactly the sets
$\lambda \nsq^i\K_i + \mu$,
where $\lambda \in \F^*$ is a square, and $\mu\in \F$. 
\end{enumerate}
\end{thm}

\begin{proof}
Denote by $K$ the minimal subquasigroup
generated by $0$ and $1$. \lref{s3} implies that if $K$ contains a nonsquare, 
then $K = \K= \K_0=\K_1$. The other possibility is that $K$ consists of
squares only. Then $K = \K_0$.

Suppose that $\lambda\in \F^*$ and $\mu \in \F$.
If $\lambda$ is a square, then $x\mapsto \lambda x + \mu$
is an automorphism of $Q$, by \pref{f3}(iii). That makes $\lambda K + \mu$
a minimal subquasigroup of $Q$. If $K$ contains a
nonsquare $\nsq$, then $\nsq K = K$ and $\lambda K +\mu =\lambda\nsq K +\mu$. 

Let $S$ be a minimal subquasigroup of $Q$ that contains $0$.
By \pref{f3}(ii) no other subquasigroups need to be considered.

Suppose there exists $s\in S$ and $\xi\in K$ such that $\chi(s)=\chi(\xi)\ne0$.
Let $\lambda=s/\xi$ and note that $x\mapsto \lambda x$ is an
automorphism of $Q$ that maps $\{0,\xi\}$ to $\{0,s\}$. By minimality,
the former set generates $K$, while the latter set generates $S$.
Therefore $\lambda K=S$. Such a $\lambda$
always exists if $K$ contains a nonsquare.

Suppose that $K=\K_0$ consists of squares only.
If $S$ contains a nonzero square, then, as we have proved, 
$S= \lambda \K_0$, where $\lambda$ is a square. In such a case all elements
of $S$ are squares. What remains to be characterised are those
minimal subquasigroups $S$ where $0\in S$ and all nonzero elements
are nonsquares. 

The mapping $x\mapsto x\nsq$ yields an isomorphism $Q_{a,b}\cong Q_{b,a}$
and sends $S$ to $S\nsq$. Applying the earlier part of the proof
to $Q_{b,a}$, yields $S\nsq = \lambda \K_1$, where $\lambda$ is a square.
\end{proof}

\begin{thm}\label{s5}
Let $Q =Q_{a,b}$ be a quadratic quasigroup over $\F$
that is not a Netto quasigroup. Let $\K$, $\K_0$ and $\K_1$ be the
subfields generated by $\{a,b\}$, $\{a\}$ and $\{b\}$, respectively.
Suppose that there exists
a $2$-generated subquasigroup of $Q$ that is neither
trivial nor minimal. Then all such subquasigroups are exactly  the
sets $\lambda \K + \mu$, where $\lambda \in \F^*$ and
$\mu \in \F$. Furthermore, each of $-1$, $a$, $b$, $1-a$ and $1-b$
is a square in $\F$. There are two possibilities:
\begin{itemize}
\item[(i)] $\K_1$ consists of squares only and $\K_0$ contains
a nonsquare. In this case $\K$ is generated,
as a subquasigroup, by $\{0,s\}$ where $s\in \K$, if and only if
$\chi(s) = 1$.
In particular, $\K$ is generated by $\{0,1\}$. The minimal subquasigroups
of $Q$ are exactly the sets $\nsq \K_1 + \mu$, where $\mu,\nsq \in \F$
and $\chi(\nsq) = -1$.
\item[(ii)] $\K_0$ consists of squares only and $\K_1$ contains
a nonsquare. In this case $\K$ is generated, as a subquasigroup,
by $\{0,\nsq\}$ where $\nsq \in \K$, if and only if $\chi(\nsq) = -1$.
The minimal subquasigroups of $Q$ are exactly the sets $s\K_0 + \mu$,
where $\mu,s \in \F$ and $\chi(s) = 1$. 
\end{itemize}
\end{thm}

\begin{proof}
Let $S$ be a 2-generated subquasigroup
that is not minimal. We shall first investigate the situation
when $S$ is generated by $0$ and $1$. The treatment 
is divided into a sequence of claims. The case of general $S$
is considered at the end of the proof.

\begin{claim}\label{cm:es1}
\text{$S$ is generated by any set $\{0,s\}$, where $s\in S$ is
a nonzero square.}
\end{claim}

Consider the automorphism
$x\mapsto sx$, which sends $\{0,1\}$
to $\{0,s\}$. The claim follows from \lref{s1}. 

\begin{claim}\label{cm:es2}
\text{$S$ contains a nonsquare.}
\end{claim}

Assume the contrary. Then there exists a minimal subquasigroup
of $Q$ that is contained in $S$ and is generated by 
$\{0,c\}$, $c$ a square. That cannot happen, by \cmref{cm:es1}.

\begin{claim}\label{cm:es3}
(a) \text{$\chi(x)\ge 0$ for all $x\in \K_1$; and }\label{es3}\\
(b) 
A set $M$ containing $0$ is a minimal
subquasigroup of $Q$ if and only if
there exists $\nsq\in \F$ such that
$M=\nsq\K_1 \text{ and }\chi(\nsq) = -1$. 
\end{claim}

There must be some minimal subquasigroup $M_0$ of $Q$ satisfying
$0 \in M_0\subset S$.  Every nonzero element of $M_0$ is a nonsquare,
by \cmref{cm:es1}.  Consider a nonzero $\zeta \in M_0$ and note that
$\zeta\m M_0$ consists of squares only.  The isomorphism $Q_{a,b}\cong
Q_{b,a}$, $x\mapsto \zeta\m x$, sends $M_0$ to the minimal
subquasigroup of $Q_{b,a}$ generated by $0$ and $1$.  By \lref{s3},
$\zeta\m M_0 = \K_1$ and hence $M_0 = \zeta \K_1$ and \cmref{cm:es3}(a)
holds. For any square $c\in \F^*$, we know that $c\zeta\K_1$ is a
minimal subquasigroup of $Q$, by \pref{f3}(iii), which proves the
``if'' part of \cmref{cm:es3}(b).

For the converse direction, consider a minimal subquasigroup $M\ni0$.
If $M$ contains a nonsquare $\zeta$, then $M=\zeta \K_1$, since we already
know that $\zeta \K_1$ is a minimal subquasigroup. If $M$ contains
a nonzero square $c$, then $\{0,1\}\subseteq c^{-1}M$. That would imply
that $S\subseteq c^{-1}M$, which is impossible because $S$ is not minimal,
but $c^{-1}M$ is minimal. 

\begin{claim}\label{cm:es4}
\text{All of the elements $-1$, $a$, $b$,
$1-a$ and $1-b$ are squares.}
\end{claim}

Suppose that $-1$ is a nonsquare. By \pref{f4} then there
exists an automorphism of $Q$ that sends $0$ to $\nsq$ and $1$ to $0$,
for each nonsquare $\nsq \in S$.
That contradicts \cmref{cm:es3}(b), given that $S$ is not minimal.
Thus $-1$ is a square. We already
know from \cmref{cm:es3}(a) that $b$ is a square. Hence $a$ is a square too. 
To see that $1-a$ and $1-b$
are squares as well, consider the opposite quasigroup,
using \pref{f3}(v).

Set
\begin{equation}\label{es5}
\begin{gathered}
N_1=\{x\in S:\chi(x) = -1\},\ S_1=\{x\in S:\chi(x)=1\},\\
N_0 = N_1\cup\{0\} \text{ and } S_0 = S_1\cup\{0\}.
\end{gathered}
\end{equation}

\begin{claim}\label{cm:es6}
\text{$S$ is the least subquasigroup containing $N_0$.}
\end{claim}

If $c\in S_1$, then $cS = S$, by \lref{s1}. This implies that
$\nsq S_1 \subseteq S$, for every $\nsq \in N_1$.
Since $\nsq S_1 \subseteq N_1$, we have 
$|S_1| = |\nsq S_1| \le |N_1|$.
Therefore $|N_0| > |S|/2$, which means 
that $N_0$ cannot be contained in a proper subquasigroup of $S$.

\begin{claim}\label{cm:es7}
\text{If $c\in \K_1\setminus\{0\}$, then $cS = S$. In particular, $\K_1\le S$.}
\end{claim}

By \cmref{cm:es6}, $N_0$ generates $S$. If $\nsq \in N_1$ and
$c\in\K_1\setminus\{0\}$, then $c\nsq$ is in $\nsq\K_1$, which is
minimal by \cmref{cm:es3}(b), and hence generated by $0$ and $\nsq$.
But $\{0,\nsq\}\subseteq S$, so $c\nsq\in S$, and it then follows
from \cmref{cm:es3}(a) that $cN_1=N_1$. Hence also $cS = S$, by \lref{s1}.

\begin{claim}\label{cm:es8}
\text{If $s\in S_0$, then $s+S = S$.}
\end{claim}

Choosing $c=-1$ yields $-S = S$, by \cmref{cm:es7}. The quasigroup
$S$ is thus generated by $\{0,-s\}$, in view of \cmref{cm:es1} and
\cmref{cm:es4}. Denote by $\psi$ the automorphism $x\mapsto x+s$.
Since $\psi(0) = s$ and $\psi(-s) = 0$, we have $S=\psi(S)
=s+S$, by \lref{s1}.

\begin{claim}\label{cm:es9}
$S=\K$.
\end{claim}

Both $a$ and $b$ are squares that belong to $S$. Since $cS = S$
whenever $c\in S$ is a square, $a^ib^j \in S_0$ for all integers
$i\ge 0$ and $j\ge 0$. A sum of such elements belongs to $S$
by \cmref{cm:es8}. By \lref{s2}, $\K$ is a subquasigroup and $\K\subseteq S$. 
Since $\{0,1\}\subseteq\K$, we must have $S \subseteq \K$.

Recall that $\K_0$ denotes the subfield generated by $a$. If $\K_0$ consists of
squares only, then $\K_0$ is a subquasigroup of $Q$. That cannot be,
by \cmref{cm:es2}, since $\{0,1\}\subseteq\K_0$.

Let $S'$ be any subquasigroup of $Q$ with $0\in S'$. 

\begin{claim}\label{cm:es10}
The subquasigroup
$S'$ is minimal if and only if $S'=\nsq \K_1$ for a nonsquare $\nsq$.
Also $S'$ is $2$-generated nonminimal if and only if $S'=\lambda \K$,
where $\lambda \in \F^*$.
\end{claim}

The first equivalence corresponds to \cmref{cm:es3}(b). Now $S = \K$ and
$S$ contains some nonsquare $\nsq$. 
For any $\lambda\in\F^*$, one of $\lambda S$ or
$(\lambda\nsq\m)S$ is a subquasigroup isomorphic to $S$, by \pref{f3}(iii).
However, $(\lambda\nsq\m)S=(\lambda\nsq\m)\nsq S=\lambda S$
because $S=\K$ is a field. It follows that
$\lambda \K$ is a 2-generated nonminimal subquasigroup of $Q$ for all
$\lambda\in\F^*$.

Next suppose that $S'\ni0$ is a nonminimal subquasigroup generated by
elements $0$ and $\lambda$.  Let $S''=\lambda\m S'$. Note that
$\lambda$ must be a square since otherwise $0$ and $\lambda$ would
generate the minimal subquasigroup $\lambda\K_1$, by \cmref{cm:es3}(b). So
$x\mapsto\lambda\m x$ is an isomorphism $S'\cong S''$ and it maps
$0,\lambda$ to $0,1$. Hence $S''=S=\K$ and $S'=\lambda\K$.

It remains to consider the case when $S'\ni0$ is a nonminimal
subquasigroup generated by two general elements $x$ and $y$.
By \pref{f3}(ii), $S'$ is isomorphic to $S'''=S'-x$. Now $S'''$ is
a nonminimal subquasigroup generated by $0$ and $y-x$, so by the
previous case $S'''=(y-x)\K$. But $0\in S'=(y-x)\K+x$ so $-x(y-x)\m\in\K$.
But $\K$ is closed under addition, so
$S'=(y-x)(\K-x(y-x)\m)+x=(y-x)\K$, which completes the proof of \cmref{cm:es10}.

\medskip

Let us now turn to the general case. Recall that $S$ is a
subquasigroup generated by two distinct elements that is not
minimal, but we no longer assume it is generated by $\{0,1\}$.
Clearly it may be assumed that there exists $\xi \in \F^*$
such that $S$ is generated by $\{0,\xi\}$. If $\xi$ is a square, then
the subquasigroup $S\xi^{-1}$ generated by $0$ and $1$ is also not minimal, and
that allows us to use the characterisation developed above.  Hence, we
may suppose that $\xi$ is a nonsquare in $\F$. Then $S\xi\m$ is a
subquasigroup of $Q_{b,a}$ that is generated by $0$ and $1$, and is not
minimal. Therefore $S\xi\m = \K$.  Since $\K$ contains a nonsquare,
say $\nsq$, we have $S=\K\xi = \K c$, where $c = \nsq\xi$ is a
square. Hence $\K$ is a subquasigroup of $Q$ that is $2$-generated,
but not minimal.

In $Q_{b,a}$ the proper minimal subquasigroups of $\K$ containing $0$
are exactly all of the sets $\K_0\xi$, where $\xi$ is a
nonsquare. Furthermore, the field $\K_0$ consists of squares only. The
minimal subquasigroups of $Q$ that include $0$ are thus equal to
$s\K_0$, where $\chi(s) = 1$.
\end{proof}

Theorems~\ref{s4} and~\ref{s5} describe the structure
of minimal subquasigroups in all quadratic quasigroups that are not
Netto quasigroups. The subfield $\K$ generated by $a$ and $b$ has a clear
structural meaning in all these quasigroups except those that
are described by \tref{s4}(ii) and fulfil $\K_0\ne \K_1$. 
Note that $\K_0=\K_1$ for all affine quasigroups $Q_{a,a}$, and hence
\tref{s5} implies that all 2-generated subquasigroups of $Q_{a,a}$ are minimal.

\subsection*{Affine lines and semilinear mappings} 
Suppose that $Q=Q_{a,b}$ is a quadratic quasigroup
that contains minimal subquasigroups of two different orders.
By Theorems~\ref{s4} and~\ref{s5} this implies that $\K_0\ne \K_1$,
where $\K_0$ is the least subfield containing $a$, and $\K_1$ the least
subfield containing $b$, and that both these subfields consist of squares
only. The existence of a subfield consisting only of squares implies
that $|\F| = q^2$ and that $\K_0 \cup \K_1\subseteq\F_q$. 

For distinct $x,y\in \F$ denote by $S(x,y)$ the subquasigroup $S$
generated by $x$ and $y$. Put $\lambda = y -x$. As follows from \tref{s4}, 
$S = \lambda\K_0 + x$ if $\chi(\lambda) = 1$ and $S = \lambda\K_1 + x$
if $\chi(\lambda) = -1$. 

Call a set $B\subseteq \F$ \emph{saturated} if there exists 
$i \in \{0,1\}$ such that for any distinct $x,y \in B$
both $|S(x,y)|=|\K_i|$ and $S(x,y)\subseteq B$ are true.

It follows from the above description of subquasigroups $S(x,y)$ that
any set $\lambda \F_q + \mu$ is saturated, provided $\lambda \in \F^*$
and $\mu \in \F$. This may be converted: 

\begin{lem}\label{s6}
A $q$-element set $B\subseteq \F$ is saturated if and only if there
exist $\lambda \in \F^*$ and $\mu \in \F$ such that $B = \lambda \F_q
+ \mu$.
\end{lem}

\begin{proof} Let $B$ be saturated. By the definition of a saturated set 
there exists $\eps \in \{-1,1\}$ such that
$\chi(y-x) = \eps$ whenever $x,y \in B$ and $x\ne y$. 
Every $q$-element set with the latter property is equal to 
a set of the form $\lambda \F_q+ \mu$, where $\lambda \ne 0$, by a theorem
of Blokhuis \cite{blok}.
\end{proof}

\begin{cor}\label{s7} Let $Q_{a,b}$ be a quasigroup over $\F$ such that  
$Q_{a,b}$ contains minimal subquasigroups of two different orders. 
Then there exists $q>1$ such that $|\F| = q^2$, and $\aut(Q)$
acts on the set of all affine lines $\lambda \F_q + \mu$,
where $\lambda \in \F^*$ and $\mu \in \F$. 
\end{cor}
\begin{proof} This follows from \lref{s6} since each automorphism
of $Q$ maps a saturated set to a saturated set.
\end{proof}

Let $K$ be a subfield of $\F$. A permutation $\sigma$ of $\F$
is said to be \emph{$K$-semilinear} if $\sigma$ is additive and
there exists $\alpha \in \aut(K)$ such that $\sigma(\lambda x) =
\alpha(\lambda)\sigma(x)$ for all $x\in \F$ and $\lambda \in K$. 

Note that if $L \subseteq K \subseteq \F$ are fields, and $\sigma$ is 
a $K$-semilinear permutation of $\F$, then $\sigma$ is  
$L$-semilinear too. 

\begin{prop}\label{s8}
Let $\psi\in \aut(Q)$, where $Q= Q_{a,b}$ is a quadratic
quasigroup defined on $\F$ that is not a Netto
quasigroup. Let $\K$ be the subfield of $\F$ generated by 
$a$ and $b$. Then there exists $\mu \in \F$ and a $\K$-semilinear 
mapping $\sigma$ such that $\psi(x) = \sigma(x) + \mu$ for all $x\in \F$.
\end{prop}

\begin{proof} If $Q=Q_{a,b}$ is $2$-generated, use \tref{e7}.
Suppose that $Q$ is not $2$-generated. If all minimal
subquasigroups of $Q$ are of the same order, put $K = \K$. 
If there is no common
order, denote by $K$ the subfield of $\F$ that is of order $q$,
where $|\F| = q^2$. Then $\psi$ maps an affine line of $K$ to
an affine line of $K$. This is true for quasigroups in which all
2-generated subquasigroups are minimal by \tref{s4} and 
\cref{s7} since $\psi$ preserves the structure of
minimal subquasigroups. The other cases follow from \tref{s5} since
$\psi$ also preserves the structure of 
2-generated subquasigroups that are not minimal.
By  The Fundamental Theorem of Affine Geometry \cite{art}  
there thus exist $\mu\in \F$ and a $K$-semilinear permutation $\sigma$ of
$\F$ such that $\psi(x) = \sigma(x)+\mu$ for all $x\in \F$.
\end{proof}

\section{Automorphisms and isomorphisms}\label{a}

This section proves Theorems~\ref{r2} and~\ref{r1}.  The proof
of \tref{r2} is done separately for the affine case, twisted case, and
all other situations. As shown in \lref{f8}, Steiner quadratic
quasigroups that are not affine are induced by Netto systems and
fulfil the condition of \tref{r1}.  As mentioned in \secref{r}, the
automorphism group of a Netto system is known \cite{rob} and conforms
with our statement of \tref{r2}.  Netto quasigroups are
thus not discussed in this section.

The proof of \tref{r2} has two parts: first we have
to verify that certain mappings are automorphisms,
and second we have to show that there are no other
automorphisms. In view of \pref{f3}, to achieve the
first goal, only the affine and twisted cases need to
be considered. For the affine case $x*y = x+a(y-x)$, so
it is clear that any $\K$-linear map $\sigma$ is
an automorphism, $\K$ being the least subfield
containing $a$. The second part of the proof of \tref{r2} for
affine quasigroups follows directly from \pref{a2}(ii) below.

Recall that $Q = Q_{a,b}$ is called twisted
if $\K$, the subfield generated by $a$, is of order
$\gamma^2$ and $b = a^\gamma$. The meaning of $\gamma$
is considered to be fixed throughout this section,
whenever $Q$ is twisted.

To see that every permutation
described in \tref{r2} is an automorphism of $Q_{a,b}$ it thus 
remains to verify the existence of automorphisms that generalise
the automorphisms induced by the 
structure of a quadratic nearfield:

\begin{lem}\label{a1}
Let $Q = Q_{a,b}$ be a twisted quadratic quasigroup over $\F$.
Then $x\mapsto \nsq x^\gamma + \mu$ is an automorphism
of $Q$ whenever $\nsq$ is a nonsquare in $\F$ and $\mu \in
\F$.
\end{lem}

\begin{proof} It may be assumed that $\mu =0$. 
Let $x,y\in \F$ be distinct elements. Put $i=0$ if $\chi(y-x)=1$,
and $i = 1$ if $\chi(y-x)=-1$. Then $x*y = x + a^{\gamma^i}(y-x)$
and 
\[ \nsq(x*y)^\gamma = \nsq x^\gamma + \nsq b^{\gamma^i}(y-x)^\gamma =
\nsq x^\gamma + b^{\gamma^i}(\nsq y^\gamma -\nsq x^\gamma)=
\nsq x^\gamma * \nsq y^\gamma,\]
since $\chi(y^\gamma-x^\gamma)=\chi((y-x)^\gamma)=\chi(y-x)$ because
$\gamma$ is odd.
\end{proof}

\begin{prop}\label{a2}
Let $Q=Q_{a,b}$ be a quadratic quasigroup over $\F$ that is not
a Netto quasigroup. Denote by $\K$ the subfield
of $\F$ that is generated by $a$ and $b$. Let $\psi\in \aut(Q)$ be such
that $\psi(0)=0$. Then:
\begin{enumerate}
\item[(i)] $\psi(1)$ is a square in $\F$ if $Q$ is neither affine nor twisted;
and
\item[(ii)] $\psi$ is $\K$-linear if $a=b$ or if $\psi(1)$ is 
a square in $\F$.
\end{enumerate}
\end{prop}

\begin{proof} Point (i) coincides with \pref{e5}. Put $\lambda = \psi(1)$.
Point (ii) will be first proved under the assumption
that all minimal subquasigroups of $Q$ are of the same order.
In this case, $\psi(\K) = \lambda \K$ by Theorems~\ref{s4}
and~\ref{s5} since a minimal subquasigroup is mapped to a minimal
subquasigroup and a $2$-generated subquasigroup is mapped to
a $2$-generated subquasigroup. The mapping $x\mapsto \lambda\m \psi(x)$
hence is an automorphism of $Q$ that sends the $2$-generated
subquasigroup $\K$ to itself. The restriction of the latter
mapping to $\K$ is $\K$-linear by \tref{e7}. Since the
mapping is $\K$-semilinear, by \pref{s8}, it has to be a $\K$-linear
automorphism of $Q$. Hence $\psi$ is $\K$-linear as well.

Let us now assume that $Q$ contains minimal subquasigroups of
two distinct sizes. Let $\alpha\in \aut(\K)$ be the automorphism
such that $\psi(cx) = \alpha(c)\psi(x)$ for all $c\in \K$.
In this case $|\F|=q^2$ for some $q>2$ and
$\psi(\F_q) = \lambda\F_q$, by \cref{s7}. Hence $x\mapsto \lambda\m \psi(x)$
is a $\K$-semilinear mapping with automorphism $\alpha$, the restriction
of which yields an automorphism
of $(\F_q,*)$ such that $x*y = x+a(y-x)$ for all $x,y \in \F_q$. Therefore
$\alpha(a) = a$, by the earlier part of this proof. To finish it suffices
to show that $\alpha(b) = b$ too. To that end, choose a nonsquare
$\nsq \in \F$ and note that $\psi(\nsq \F_q) = \lambda' \nsq \F_q$
for a square $\lambda'$, by \cref{s7}. The mapping 
$x\mapsto \nsq\m (\lambda')\m \psi(\nsq x)$ is an automorphism 
of $Q_{b,a}$ that sends $\F_q$ to $\F_q$ and is associated,
as a $\K$-semilinear mapping, with the automorphism
$\alpha \in \aut(\K)$. Since $b$ is now in the position of $a$,
we must have $\alpha(b) = b$.
\end{proof}

For non-affine cases the following fact is of crucial importance.

\begin{lem}\label{a3}
Suppose that $\sigma\colon\F\to \F$ is bijective and additive. If 
\begin{equation}\label{ea1}
  \chi(v-u)=\chi(\sigma(v)-\sigma(u))\text{ for all }u,v\in \F,
\end{equation}
then there exist $\alpha \in \aut(\F)$ and $\lambda
\in \F$ such that $\sigma(x) = \lambda\alpha(x)$ for each
$x\in \F^*$ and $\chi(\lambda) = 1$.
\end{lem}

\begin{proof}
The set of all additive permutations $\sigma$ of $\F$ that satisfy
\eref{a1} forms a group. This group contains the mapping
$x\mapsto\lambda x$ for each $\lambda \in \F^*$, with $\chi(\lambda) = 1$. 
Since $\chi(\sigma(1))=\chi(\sigma(1)-\sigma(0))=\chi(1)=1$, it suffices to 
prove the statement under the assumption that $\sigma(1) = 1$.
However, if $\sigma$ is a permutation of $\F$ such that
$\sigma(0) = 0$, $\sigma(1) = 1$ and \eref{a1} holds, then 
$\sigma\in \aut(\F)$ as shown by Carlitz \cite{carl}.
\end{proof}

\begin{lem}\label{a4}
Let $Q=Q_{a,b}$ be a non-affine quadratic quasigroup over $\F$
that is not a Steiner quasigroup.
Denote by $\K$ the least subfield of $\F$ containing
$a$ and $b$. If $\sigma \in \aut(Q)$ is $\K$-linear, then
$\sigma$ satisfies \eref{a1}.
\end{lem}

\begin{proof} If $\chi(v-u) = 1$, then
$\sigma(u*v)=\sigma(u+a(v-u)) = \sigma(u)+a(\sigma(v)-\sigma(u))$.
Hence $\sigma(u*v) = \sigma(u)*\sigma(v)$ implies that
$\chi(\sigma(v)-\sigma(u)) = 1$. The argument when $\chi(v-u) = -1$ is
similar.
\end{proof}

\begin{lem}\label{a5}
Let $Q = Q_{a,b}$ be a non-affine quadratic quasigroup over $\F$ that
is neither Steiner nor twisted. Let $\K$ be the subfield generated
by $a$ and $b$. Then $\aut(Q) =\agtl_1(\F\midl \K)$.
\end{lem}

\begin{proof} 
Let $\psi \in \aut(Q)$ be such that $\psi(0) = 0$. Then $\psi$
is $\K$-linear by \pref{a2}, and is a square scalar multiple of 
some $\alpha \in \aut(\F)$ by Lemmas~\ref{a4} and~\ref{a3}. 
\end{proof}

\begin{lem}\label{a6}
Let $Q = Q_{a,b}$ be a twisted quasigroup over $\F$.
Then $\aut(Q) = \agltw_1(\F\midl \K)$.
\end{lem}

\begin{proof} 
Let $\psi \in \aut(Q)$ be such that $\psi(0) = 0$. If $\psi(1)$
is a square, then $\psi\in \agtl_1(\F\midl \K)$, by the same
argument as in the proof of \lref{a5}. Suppose that $\nsq = \psi(1)$
is a nonsquare and compose $\psi$ with the mapping 
$\sigma\colon x\mapsto (\nsq\m)^\gamma x^\gamma$. Then $\sigma \psi \in 
\aut(Q)$ by \lref{a1} and $\sigma\in \agltw_1(\F\midl \K)$. 
Since $\sigma\psi(1) = 1$ is a square, $\sigma\psi \in \agtl_1(\F\midl \K)$.
These facts imply that $\psi$ belongs to $\agltw_1(\F\midl \K)$.
\end{proof}

This finishes the proof of \tref{r2}. What follows is a proof
of \tref{r1}.

\begin{proof} As mentioned at the beginning of this section,
it may be assumed that $a=b$ if $Q = Q_{a,b}$ is a Steiner quasigroup.
The subfield generated by $a$ and $b$ is denoted by $\K$. 

Let $\nsq \in \F$ be a nonsquare. Since $x\mapsto \nsq x$ maps isomorphically
$Q_{a,b}$ to $Q_{b,a}$, and $Q_{c,d}$ to $Q_{d,c}$, certain assumptions
may be made. By \tref{s5}, it may be assumed that if $Q$ possesses
a $2$-generated subquasigroup that is not minimal, then $\K$ carries
one such subquasigroup. It may also be assumed that there exists an
isomorphism
$\sigma\colon Q_{a,b}\cong Q_{c,d}$ that sends $0$ to $0$ and $1$
to a square in $\F$. Since scalar multiplication by a nonzero square is
an automorphism of $Q_{c,d}$, it may be assumed, in fact, that
$\sigma(0) = 0$ and $\sigma(1) = 1$. 

Suppose that $Q$ does not contain minimal subquasigroups of distinct
orders. Then $\K$ coincides with the subquasigroup generated
by $0$ and $1$, by Theorems~\ref{s4} and~\ref{s5}. By these
theorems $\sigma(\K) = \K$ since $\sigma(\K)$ is a subquasigroup
of $Q_{c,d}$ that is generated by $0$ and $1$. Applying \tref{e8}
to the restriction of $\sigma$ to $\K$ yields 
$\alpha\in \aut(\K)$ such that $c = \alpha(a)$ and $d = \alpha(b)$.
This suffices, since $\alpha$ may be extended to an automorphism 
of $\F$.

Suppose now that $Q$ contains minimal subquasigroups of two different
orders. Then $|\F| = q^2$ for some $q>2$,
and in $Q$ there exists a unique saturated
set of order $q$ that contains both $0$ and $1$. This set is 
equal to $\F_q$. Since $\sigma$
maps a saturated set to a saturated set, it must be that
$\sigma(\F_q) = \F_q$. The orders of minimal subquasigroups contained
in $\F_q$ are thus the same in both $Q_{a,b}$ and $Q_{c,d}$.
By interpreting an affine line $\lambda \F_q +\mu$ as a saturated
set we therefore obtain that for all $\lambda \in \F^*$ and $\mu \in \F$
there exist $\lambda' \in \F^*$ and $\mu'\in \F$ such that
$\chi(\lambda) = \chi(\lambda')$ and $\sigma(\lambda \F_q+\mu)
= \lambda'\F_q + \mu'$. This means that $\sigma$ fulfils condition
\eref{a1}. Since $\sigma(0) = 0$ and $\sigma(1) =1$, the theorem
of Carlitz \cite{carl} implies that $\sigma \in \aut(\F)$. 
To get $\sigma(a) = c$ consider the quasigroup product of $0$ and
$1$. To obtain $\sigma(b) = d$ multiply $0$ and $\nsq$.
\end{proof}
 
\section{Concluding comments}

The theory developed in this paper should prove to be useful in the many
applications of quadratic quasigroups
\cite{AW22,DW21,Eva92,Eva18,GW20,cyclatom}.
In \cite{cyclatom}, methods were developed for distinguishing isomorphism
classes of quasigroups generated from cyclotomic orthomorphisms. In the
quadratic case, this is now a very simple task, given \tref{r1}. It would
be of interest to develop similar methods for quasigroups generated from
other cyclotomic orthomorphisms. Another feature of \cite{cyclatom} is that
commutative, semisymmetric and totally symmetric quasigroups played a prominent
role, as they did in the current work. It thus should prove useful to
have the characterisations in \tref{f10}.

It was mentioned in the introduction that this paper grew out of a
need in \cite{DW21} to understand when quadratic quasigroups are
isomorphic to each other. In that paper, we showed that,
asymptotically, a nonzero constant fraction of the choices for the
pair $(a,b)$ result in $Q_{a,b}$ having a special property called
maximal non-associativity. We conjectured that removing isomorphs
would not reduce the demonstrated number of examples by more than a
factor proportional to $\log|\F|$. Thanks to \tref{r1}, we now know
this conjecture to be true, since the automorphism group of $\F$ has
order $O(\log|\F|)$.

\end{document}